\documentclass [reqno] {amsart}

\usepackage{amsfonts}
\usepackage{amssymb}

\usepackage{color}
\usepackage{hyperref}

\newtheorem{theorem}{Theorem}[section]
\newtheorem{lemma}[theorem]{Lemma}
\newtheorem{corollary}[theorem]{Corollary}

\theoremstyle{definition}

\theoremstyle{remark}

\numberwithin{equation}{section}

\newcommand{\B}{\mathbf{B}_i}

\begin{document}
	
	\title[An open mapping theorem]{An open mapping theorem for  
		nonlinear operator equations associated with elliptic complexes
	}

	\author[A. Polkovnikov]{Alexander Polkovnikov}
	
	\address{Siberian Federal University,
		Institute of Mathematics and Computer Science,
		pr. Svobodnyi 79,
		660041 Krasnoyarsk,
		Russia}
	
	\email{paskaattt@yandex.ru}
	
	
	
	
	
	
	\begin{abstract}
		Let $ \{A^i,E^i\} $ be the elliptic complex on a $ n $-dimensional smooth 
		closed Riemannian
		manifold $X$ with the first order differential operators $ A^i $ and smooth vector 
		bundles $ E^i $ over $X$. We consider nonlinear operator equations, associated with the 
		parabolic differential operators $\partial_t +   \Delta^{i}  $, generated by the Laplacians $\Delta^{i} 	$ of the complex $ \{A^i,E^i\}$, in special 
		Bochner-Sobolev functional spaces.
		We prove that under reasonable assumptions regarding the nonlinear term  
		the Frech\'et derivative $ \mathcal{A}_i' $ of 
		the induced nonlinear mapping is continuously invertible  and the map $ 
		\mathcal{A}_i $ is open and injective in chosen spaces.
	\end{abstract}

	\maketitle
	
	\section*{Introduction}
	
	Let $X$ be a Riemannian $ n $-dimensional smooth compact closed manifold and $ E^{i} $ be 
	smooth vector bundles over $X$. Denote by $ C^{\infty}_{E^{i}}(X) $  the space of all smooth 
	sections of the bundle $ E^{i} $. Consider an  
	elliptic complex of the first order differential operators $ A^i $ on $X$,
	\begin{equation}\label{key}
		0\longrightarrow C^{\infty}_{E^{0}}(X)\xrightarrow{A^0} C^{\infty}_{E^{1}}(X)
		\xrightarrow{A^1} \cdots \xrightarrow{A^{N-1}} C^{\infty}_{E^{N}}(X)\longrightarrow 0.
	\end{equation}
	where $ A^i\circ A^{i-1} \equiv 0$. 	 In this case it is equivalent to say 
	that the Laplacians	$  \Delta^{i} = (A^i)^* A^{i} + A^{i-1} (A^{i-1})^*$, 
	$ i=0, 1, \dots, N $, of the complex are the second order strongly elliptic differential 
	operators on $X$ where operator $ (A^i)^* $ is formal adjoint to $ A^i $ (see \eqref{key655}  below). For $ i<0 $ and $ i\geq N $ we assume that $ A^{i} = 0 $.
	
	Inspired by typical nonlinear problems of the Mathematical Physics, see,
	for instance \cite{Lion69}, \cite{Temam79}, we consider a family of nonlinear 
	parabolic equations, associated with the complex $ \{A^i,E^i\}$. 
	With this purpose, we denote by $ M_{i,j} $ two bilinear bi-differential operators of zero 
	order (see \cite{Bourbaki} or \cite{Tarkhanov95}),
	\begin{equation}\label{eq1}
		M_{i,1}(\cdot,\cdot): (E^{i+1}, E^{i})\to E^{i},\quad M_{i,2}(\cdot,\cdot): (E^{i}, E^{i})\to E^{i-1}.
	\end{equation}
	We set for a differentiable section $ v $ of the bundle $ E^{i} $
	\begin{equation}\label{eq2}
		N^{i}(v) =: M_{i,1} ( A^{i}v, v) +  A^{i-1} M_{i,2}(v, v).
	\end{equation}
	
	Let now $ X_T $ be a cylinder, $ X_T = X \times [0,T]$, where the
	time $ T>0 $ is finite. Then, for any fixed positive number $ \mu $  
	the operators $ \partial_t + \mu \Delta^{i} $ are parabolic on $  X \times (0,+\infty) $ (see \cite{Eidelman}).
	
	Consider the following initial problem: given section $ f $ of the induced bundle $ E^{i} (t) $ (the variable $ t $ enters into this bundle as a 
	parameter) and section $ u_0 $ of the bundle $ E^{i} $, find a section $ u $ of the induced bundle $ E^{i} (t) $ and a section $ p $ of the induced bundle $ E^{i-1} (t) $ such that
	\begin{equation}\label{eq3}
		\begin{cases}
			\partial_t u + \mu \Delta^{i} u + N^{i}(u) + A^{i-1} p = f& \text{in } X \times (0,T),\\
			(A^{i-1})^* u =0 & \text{in } X \times [0,T],\\
			(A^{i-2})^* p =0 & \text{in } X \times [0,T],\\
			u(x,0) = u_0& \text{in } X,
		\end{cases}
	\end{equation}
	Recently such a problem was considered in the weighted H\"older spaces 
	over ${\mathbb R^n} \times [0,T]$, where the weights provide prescripted 
	asymptotic behaviour at the infinity with respect to the space variables 
	(see \cite{ShlapunovPar});   it was proved that image of the operator 
	$\mathcal{A}_i$, induced by \eqref{eq3}, is open in these spaces.
	
	If $ n=2 $ or $n=3$ and complex \eqref{key} 
	is the de Rham complex, $ \{A^i,E^i\}  = \{d^i,\Lambda^i\} $ with 
	the exterior differentials acting on sections of the bundle of the exterior 
	differential forms on $X$, then for $i=1$ and a suitable choice 
	of the nonlinear term we may treat \eqref{eq3} 
	as the initial problem for the well known Navier-Stokes equations for incompressible fluid over the manifold $X$ (see, for instance, \cite{Tarkhanov19}). 
	Note that if $i=1$ then the equation with  respect to $p$ is actually 
	missing because $(A^{-1})^* = 0$.
	
	Of course, there are plenty of  papers devoted to Navier-Stokes equations on 
	Riemannian manifolds beginning from the pioneer work \cite{Ebin70}. However we will not concentrate our efforts on aspects of Hydrodynamics.

	In contrast to \cite{ShlapunovPar} we consider this problem in special Sobolev-Bochner type spaces, cf. \cite{ShT20} 
	for the de Rham  complex on the torus ${\mathbb T}^3$ in the case where $i=1$. 
	Namely, using the standard tools, 
	such as the interpolation Gagliardo-Nirenberg inequalities, Gronwall type lemmas and 
	Faedo-Galerkin approximations, see, for instance, \cite{Lion69}, \cite{Temam79}, 
	we show that suitable linearizations of our problem has 
	one and only one solution and nonlinear problem (\ref{eq3}) has at most 
	one solution in the constructed spaces. Applying the
	implicit function theorem for Banach spaces we prove that the image of $\mathcal{A}_i$ is 
	open in selected spaces (thus, obtaining the so-called "open mapping 
	theorem"). However, we do not discuss here much more delicate question 
	on the existence of solution to nonlinear 
	problem (\ref{eq3}) aiming at the maximal generality of the nonlinear term 
	$N^{i}(u)$ for the open mapping theorem to be true. It is worth to note that 
	even for existence of weak (distribultional) solutions to (\ref{eq3}) one has to impose rather restrictive but reasonable additional assumptions on the nonlinearity $N^{i}$.

	\section{Functional spaces and basic inequalities}
	
	We need to introduce appropriate functional spaces. Namely, let $ dx $ be a volume form on $ X $ and a $ (\cdot,\cdot)_{x,i} $ denotes a Riemannian metric in the fibres of $ E^{i} $. As usual, we equip each bundle $ E^{i} $ with
	a smooth bundle homomorphism $ *_i : E^{i} \to E^{i*} $ defined by $ <*_i u,v>_{x,i} = (v,u)_{x,i}$ for all $ u,v\in E^{i}$. Then we can to consider the space $ C^{\infty}_{{i}}(X) $  with the unitary structure
	\begin{equation*}
		(u,v)_{i} = \int_X (u,v)_{x,i}\,dx
	\end{equation*}
	and the Lebesgue space $ L^{2}_{{i}}(X) $  with the norm $ \|u\|_i = \sqrt{(u,u)_{i}}$. In this case the formal adjoint to $ A^i $ operator $ (A^i)^*:C^{\infty}_{E^{i+1}}(X)  \to C^{\infty}_{E^{i}}(X) $ is defined in  the following way 
	for the sections $ u\in  E^{i+1} $ and $ v\in  E^{i} $:
	\begin{equation}\label{key655}
		\left( (A^i)^* u,v\right)_{i}:=   \left( u, A^i v\right)_{i+1}.
	\end{equation}
	
	Let  $ \{U_l\}_{l=1}^N $ be a finite open cover of $ M $ by coordinate neighborhoods over  which  $ E^i $  is  trivial and $ \{\psi_l\}_{l=1}^N $ a corresponding partition of unity,
	\[
	\psi_l\in C_{{i}}^\infty(X),\quad \psi_l(x)\geq 0,
	\]
	\[
	\mbox{supp}\, \psi_l \subset U_l,\quad \sum_{l=1}^N\psi_l(x) \equiv 1\text{ on } X.
	\]
	As usual, denote by $ W^{s}_{i,p}(X) $, $s\in\mathbb{Z}_+ $, $1\leq p\leq \infty$, the Sobolev space under the smooth vector bundle $ E^{i} $ (see, for instance, \cite{Agronovich15}). It is a Banach space of sections $ u \in L^{p}_{{i}}(X) $ with the norm
	\begin{equation*}
		\|u\|_{W^{s}_{i,p}(X)} =
		\left(  \sum\limits_{l=1}^N\sum\limits_{|\alpha| \leq s} \| \partial^\alpha (\psi_l u) \|^p_{L^p(\mathbb R^n)}\right) ^{1/p},
	\end{equation*}
	In particular, it is a Hilbert space for $ p=2 $, we denote it by $ H^{s}_{{i}}(X):= W^{s}_{i,2}(X)$.
	
	Now, let $ \mathcal{H}^i $ be the so-called 'harmonic space' of
	complex (\ref{key}), i.e.
	\begin{equation}\label{eq20}
		\mathcal{H}^i = \left\lbrace u \in C^{\infty}_i(X): A^i	u = 0 \mbox{ and } (A^{i-1})^*	u = 0\mbox{ in }X \right\rbrace,
	\end{equation}
	and $ \Pi^i $ be the orthogonal projection from $ L^2_i (X) $ onto $ \mathcal{H}^i $.
	
	The following standard Hodge theorem for elliptic complexes 
	plays an essential 
	role for our investifations.
	
	\begin{theorem}\label{parametrix.laplas}
		Let $ 0 \leq i \leq N $, $ s\in \mathbb{Z}_+  $, $ 0 < \lambda < 1 $. Then operator 
		\begin{equation}\label{lapl}
			\Delta^{i}: H^{s+2}_i(X)\to H^s_i(X)
		\end{equation}
		is Fredholm:
		
		(1) the kernel of operator (\ref{lapl}) equals to the finite-dimensional space $ \mathcal{H}^i $;
		
		(2) given $ v\in H^s_i(X) $ there is a form $ u\in H^{s+2}_i(X) $  such that $ \Delta^{i} u = v  $ if and
		only if $ (v, h)_i = 0 $ for all $ h \in \mathcal{H}^i $;
		
		(3) there exists a pseudo-differential operator $ \varphi^i $ on $ X $ such that the operator
		\begin{equation}
			\varphi^i: H^s_i(X)\to H^{s+2}_i(X),
		\end{equation}
		induced by $ \varphi^i $, is linear bounded and with the identity $ I $ we have
		\begin{equation}\label{eq22}
			\varphi^i \Delta^{i} = I - \Pi^i\mbox{ on } H^{s+2}_i(X),\quad  \Delta^{i} \varphi^i = I - \Pi^i\mbox{ on } H^{s}_i(X)
		\end{equation}
	\end{theorem}
	\begin{proof}
		See, for instance, 
		\cite[Theorem 2.2.2]{Tarkhanov95}.
	\end{proof}
	
	For $ m\in \mathbb{N} $ we denote by $ \tilde\nabla^m_i $ an elliptic differential operator of order $ m $, connected with the complex (\ref{key}),
	\begin{equation}\label{frac.laplase}
		\tilde\nabla^m_i : = 
		\begin{cases}
			(\Delta^i)^{m/2},\quad &m\ \text{is even},\\
			(A^i \oplus (A^{i-1})^*)(\Delta^i)^{(m-1)/2},\quad &m\ \text{is odd}.\\
		\end{cases}
	\end{equation}
	It is easy to see that kernel of $ \tilde\nabla^m_i $ for all $ m\in \mathbb{N} $ equals to $ \mathcal{H}^i $.
	Then it follows from the ellipticity of (\ref{frac.laplase}), that there exist a parametrix $ \tilde\varphi^i_m $
	such that 
	\[
	I = \tilde\nabla^m_i \tilde\varphi^i_m + \Pi^i
	\]
	with identity operator $ I $ (see, for instance, \cite[Chapter 10]{Nicolaescu}).
	Then for each $ m\in \mathbb{N} $ and $1\leq p\leq \infty$ we can to consider the completion of the space $ C^{\infty}_{{i}}(X) $ by the norm 
	\begin{equation}\label{norm.sobol}
		\|u\|_{\tilde\nabla^m_i,p} = 
		\begin{cases}
			\left( \|(\Delta^i)^{m/2}u\|^p_{L^{p}_{{i}}(X)} + \|\Pi^i u\|^p_{L^{p}_{{i}}(X)}\right)^{1/p}, &m\ \text{is even},\\
			\left( \|A^i(\Delta^i)^{(m-1)/2}u\|^p_{L^{p}_{{i+1}}(X)} +\right. \\ \left. +\|(A^{i-1})^*(\Delta^i)^{(m-1)/2}u\|^p_{L^{p}_{{i-1}}(X)}+ \|\Pi^i u\|^p_{L^{p}_{{i}}(X)}\right)^{1/p},  &m\ \text{is odd}.\\
		\end{cases}
	\end{equation}
	Also we may define in a standard way the 
	linear elliptic pseudo-differential operator $ \nabla^s_i = (\Delta^i)^{s/2}$ of order  $ s\in \mathbb{R}_+ $ on the section $ E_i $, see, for instance, \cite{Agronovich94}. 
	As above, there exist a parametrix $ \varphi^i_m $ 
	such that 
	\[
	I = \nabla^m_i \varphi^i_m + \Pi^i.
	\]
	For  $u\in C^{\infty}_{{i}}(X) $ we denote by $ \|u\|_{\nabla^s_i,p} $ the norm
	\begin{equation}\label{norm.sobol1}
		\|u\|_{\nabla^s_i,p} = \left( \| \nabla^s_i u \|^p_{L^p_i(X)} + \|\Pi^i_m u \|^p_{L^p_i(X)}\right)^{1/p} .
	\end{equation}
	\begin{lemma}\label{equivalent}
		For each $ m\in \mathbb{Z}_+ $, $ s\in \mathbb{R}_+ $ and  $ 1< p< \infty $ there are constants $ \tilde C_{1} $, $ \tilde C_{2} $, $ C_{1} $ and $ C_{2} $, such that
		\begin{equation}\label{eqiv.norm.sobolev.dif}
			\tilde C_{1} \|u\|_{\tilde\nabla^m_i,p} \leq \|u\|_{W^{s}_{i,p}(X)} \leq \tilde C_{2} \|u\|_{\tilde\nabla^m_i,p}
		\end{equation}
		\begin{equation}\label{eqiv.norm.sobolev.psedif}
			C_{1,s} \|u\|_{\nabla^s_i,p} \leq \|u\|_{W^{s}_{i,p}(X)} \leq C_{2,s} \|u\|_{\nabla^s_i,p}
		\end{equation}
		for all $ u\in C^{\infty}_{{i}}(X) $.
	\end{lemma}
	\begin{proof}
		Follows immediately from  the G\aa{}rding's inequality, definitions of operators $ \tilde\nabla^m_i $, $ \nabla^s_i $ and the fact that 
		\[
		\|\nabla^s_i u\|_{L^p_i(X)} \leq c\|u\|_{W^{s}_{i,p}(X)}
		\]
		with some positive constant $ c $, see, for instance, \cite{Agronovich94} or \cite[Proposition 2.4]{Taylor}.
	\end{proof}
	Actually, it follows from (\ref{eqiv.norm.sobolev.dif}) and (\ref{eqiv.norm.sobolev.psedif}) that we can equip the space $ W^{s}_{i,p}(X) $ with the norm (\ref{norm.sobol}) or, 
	for $1<p<\infty$, (\ref{norm.sobol1}).
	
	In the sequel we need the Gronwall's Lemma. Let us recall it.
	\begin{lemma}
		\label{l.Groenwall}
		Let $A$, $B$ and $Y$ be real-valued functions defined on a segment $[a,b]$.
		Assume that
		$B$ and $Y$ are continuous
		and that
		the negative part of $A$ is integrable
		on $[a,b]$.
		If moreover $A$ is nondecreasing, $B$ is nonnegative and $Y$ satisfies the integral 
		inequality
		\begin{equation*}
			Y (t) \leq A (t) + \int_a^t B (s) Y (s) ds
		\end{equation*}
		for all $t \in [a,b]$, then
		$
		\displaystyle
		Y (t) \leq A (t) \exp \Big( \int_a^t B (s) ds \Big)
		$
		for all $t \in [a,b]$.
	\end{lemma} 
	
	Let now $V^s_i = H^{s}_{{i}}(X) \cap S_{(A^{i-1})^*} $ stand for the space of all  the sections $ u\in H^{s}_{{i}}(X) $ satisfying $ (A^{i-1})^* u = 0 $ in the sense of the distributions in $ X $. Denote by $L^2(I,H^{s}_{{i}}(X))$ the Bochner space of $L^2$-mappings
	\[
	u(t): I \to H^{s}_{{i}}(X),
	\]
	where $ I = [0,T] $, see, for instance, \cite{Lion69}.
	It is a Banach space with the norm 
	\[
	\|u\|_{L^2(I,H^{s}_{{i}}(X))}^2 = \int_0^T \|u\|^2_{H^{s}_{{i}}(X)} dt.
	\]
	We want to introduce the suitable Bocner-Sobolev type spaces, see, for instance, \cite{ShT20} for the 
	de Rham complex in the degree $i=1$ over the torus ${\mathbb T}^3$.
	As problem \ref{eq3} is inspired by the models of 
	Hydrodynamics, where $u,f,p$ represent velocity, outer force and pressure, respectively,  
	for $s\in\mathbb{Z}_+ $ denote by $ B^{k,2s,s}_{i,\text{vel}}(X_T) $ the 
	space of sections of the induced bundle $ E^i(t) $ over $ X_T $ 
	such that
	\[
	u\in C (I,V_i^{k+2s})\cap L^2(I,V_i^{k+2s+1})
	\]
	and 
	\[
	\nabla^m_i \partial^j_t u \in C (I,V_i^{k+2s-m-2j})\cap L^2(I,V_i^{k+2s+1-m-2j})
	\]
	for all $ m + 2j \leq 2s$. It is a Banach space with the norm 
	$$
	\| u \|^2_{B^{k,2s,s}_{i,\text{vel}} (X_T)}
	:=
	\sum_{m+2j \leq 2s \atop 0\leq l \leq k}
	\| \nabla^l_i \nabla^m_i \partial_t^j u \|^2
	_{C (I,{L}_i^2(X)}
	+ \| \nabla^{l+1}_i \nabla^m_i \partial_t^j u \|^2
	_{L^2 (I,{L}_i^2 (X))}.
	$$
	
	Similarly, for $s, k \in {\mathbb Z}_+$, we define the space 
	$B^{k,2s,s}_{i,\mathrm{for} }
	(X_T)$ to consist of all sections
	\[
	f\in C (I,{H}_i^{2s+k}(X)) \cap L^2 (I,{H}_i^{2s+k+1}(X))
	\]
	with the property that 
	$$
	\nabla^m_i \partial _t^j f
	\in
	C (I,{H}_i^{k+2s-m-2j}(X)) \cap L^2 (I,{H}_i^{k+2s-m-2j+1}(X))
	$$
	provided
	$
	m+2j \leq 2s.
	$
	
	If $f \in B^{k,2s,s}_{i,\mathrm{for}} (X_T)$, then actually
	$$
	\nabla^m_i \partial _t^j f
	\in
	C (I,{H}_i^{k+2(s-j)-m}(X))
	\cap
	L^2 (I,{H}_i^{k+1+2(s-j)-m}(X))
	$$
	for all $m$ and $j$ satisfying $m+2j \leq 2s$.
	We endow the spaces $B^{k,2s,s}_{i,\mathrm{for}} (X_T)$ 
	with the natural norms 
	$$
	\| f \|^2_{B^{k,2s,s}_{i,\mathrm{for}} (X_T)}
	:=
	\sum_{m+2j \leq 2s \atop 0\leq l \leq k}
	\| \nabla^l_i \nabla^m_i \partial_t^j f \|^2
	_{C (I,{L}_i^2(X)}
	+ \| \nabla^{l+1}_i \nabla^m_i \partial_t^j f \|^2
	_{L^2 (I,{L}_i^2 (X))}. 
	$$
	
	Finally, the spaces for the section $p$ are 
	$B^{k+1,2s,s}_{i-1,\mathrm{pre}} (X_T)$.  
	By definition, they consists of all sections $p$ from the space 
	$C (I,H_{i-1}^{2s+k+1}(X)) \cap 
	L^2 (I,H_{i-1}^{2s+k+2} (X))$ such that $A^{i-1} p \in B^{k,2s,s}_{i,\mathrm{for}} (X_T)$, $ (A^{i-2})^* p = 0 $ and
	\begin{equation}\label{eq.p.ort.0}
		( p,h)_{L^2_{i-1}(X)} = 0
	\end{equation}
	for all $h \in \mathcal{H}^{i-1} $
	. 
	We equip this space with the norm
	$$
	\| p \|_{B^{k+1,2s,s}_{i-1,\mathrm{pre}} (X_T)}
	= \| A^{i-1} p \|_{B^{k,2s,s}_{i,\mathrm{for}} (X_T)}.
	$$
	
	Consider now bi-differential operator
	\begin{equation}
		\B (w,v) = M_{i,1} ( A^{i}w, v) + M_{i,1} ( A^{i}v, w) +  A^{i-1}\big(  M_{i,2}(w, v)  +  M_{i,2}(v, w)\big)
	\end{equation}
	such that  
	\begin{equation}\label{eq9}
		|M_{i,1}(u,v)|\leq c_{i,1} |u|\,|v|,\quad |M_{i,2}(u,v)|\leq c_{i,2} |u|\,|v| 
		\mbox{ on } X 
	\end{equation}
	with some  positive constants $ c_{i,j} $.
	In the sequel we will always assume that \eqref{eq9} holds on.
	
	\begin{theorem}
		\label{l.NS.cont.0}
		Suppose that
		$s \in \mathbb N$,
		$k \in {\mathbb Z}_+$, $2s+k>\frac{n}{2}-1$.
		Then the mappings
		$$
		\begin{array}{rrcl}
			\nabla^m_i :
			& B^{k,2(s-1),s-1}_{{i},\mathrm{for}} (X_T)
			& \to
			& B^{k-m,2(s-1),s-1}_{{i},\mathrm{for}} (X_T),\ m\leq k
			\\[.2cm]
			\Delta^{i} :
			& B^{k,2s,s}_{{i},\mathrm{vel}} (X_T)
			& \to
			& B^{k,2(s-1),s-1}_{{i},\mathrm{for}} (X_T),
			\\[.2cm]
			\partial_t :
			& B^{k,2s,s}_{{i},\mathrm{vel}} (X_T)
			& \to
			& B^{k,2(s-1),s-1}_{{i},\mathrm{for}} (X_T),\\
		\end{array}
		$$
		are continuous.
		Besides, if 
		$w, v \in B^{k+2,2(s-1),s-1}_{{i},\mathrm{vel}} (X_T)$
		then the mappings
		
		\begin{equation}\label{eq8}
			\begin{array}{rrcl}
				\B(w,\cdot) :
				& B^{k+2,2(s-1),s-1}_{{i},\mathrm{vel}} (X_T)
				& \to
				& B^{k,2(s-1),s-1}_{{i},\mathrm{for}} (X_T),
				\\
				\B (w,\cdot) :
				& B^{k,2s,s}_{{i},\mathrm{vel}} (X_T)
				& \to
				& B^{k,2(s-1),s-1}_{{i},\mathrm{for}} (X_T),\\
			\end{array}
		\end{equation}
		are continuous, too. In particular, for all $w, v \in B^{k+2,2(s-1),s-1}_{{i}} (X_T)$ there is positive constant $c_{s,k} $, independent on $ v $ and $w$, such that  
		\begin{equation}\label{eq.B.pos.bound}
			\|  \B (w,v)\|_{B^{k,2(s-1),s-1}_{{i},\mathrm{for}} (X_T)}
			\leq c_{s,k} 
			\|w\|_{B^{k+2,2(s-1),s-1}_{{i},\mathrm{vel}} (X_T)} 
			\|v\|_{B^{k+2,2(s-1),s-1}_{{i},\mathrm{vel}} (X_T)}.
		\end{equation}
	\end{theorem}
	
	\begin{proof}
		The proof is based on the following version of Gagliardo-Nirenberg
		inequality, see \cite{Nir59} or \cite[Theorem 3.70]{Aubin82} for the Rimanian manifolds. 
		\begin{lemma}
			For all $v \in L_i^{q_0} (X) \cap L_i^{2} (X) $ such that 
			$\nabla_i^{j_0} v \in L_i^{p_0}(X)$  and $\nabla_i^{m_0} v \in L_i^{r_0} (X)$ we have 
			\begin{equation} \label{eq.L-G-N}
				\| \nabla_i ^{j_0} v \|_{L_i^{p_0} (X)} \leq C\left( \left( 
				\| \nabla_i ^{m_0} v \|_{L_i^{r_0}(X)} + \| v \|_{L_i^{2}(X)} \right)^a 
				\| v \|^{1-a}_{L_i^{q_0} (X)}  +  \| v \|_{L_i^{2} (X)}\right) 
			\end{equation}
			with  a positive constant $C= C^{(n)}_{j_0,m_0,s_0} (p_0,q_0,r_0)$ independent on $v$, where 
			\begin{equation} \label{eq.L-G-N.cond}
				\frac{1}{p_0} = \frac{j_0}{n} + a \left(\frac{1}{r_0} -\frac{m_0}{n} \right) + 
				\frac{(1-a)}{q_0} \mbox{ and } \frac{j_0}{m_0} \leq a \leq 1,
			\end{equation}
			with the following exceptional case: if $1<r_0< +\infty$ and $m_0-j_0-n/r_0$ is a non-negative integer then the 
			inequality is valid only for $\frac{j_0}{m_0} \leq a < 1$.
		\end{lemma}
		\begin{proof}
			Indeed, under the hypothesis of Lemma, we have from (\ref{eqiv.norm.sobolev.dif}) and (\ref{eqiv.norm.sobolev.psedif}) 
			\begin{equation}\label{gal1}
				\| \nabla_i ^{j_0} v \|_{L_i^{p_0} (X)} \leq c_1\|v\|_{W^{j_0}_{i,p_0}(X)} \leq
				c_2\|v\|_{\tilde\nabla^{j_0}_i,p_0} \leq
			\end{equation}
			\[
			\begin{cases}
				c_3\left( \|(\Delta^i)^{j_0/2}u\|_{L^{p_0}_{{i}}(X)} + \|\tilde\Pi^i_{m_0} u\|_{L^{p_0}_{{i}}(X)}\right) ,\quad &\text{if}\ j_0\ \text{is even},\\
				c_3\left( \|A^i(\Delta^i)^{(j_0-1)/2}u\|_{L^{p_0}_{{i+1}}(X)} + \right. \\  \left. +\|(A^{i-1})^*(\Delta^i)^{(j_0-1)/2}u\|_{L^{p_0}_{{i-1}}(X)}+ \|\tilde\Pi^i_{m_0} u\|_{L^{p_0}_{{i}}(X)}\right) ,\quad&\text{if}\ j_0\ \text{is odd}.\\
			\end{cases}
			\]
			with positive constants $ c_1 $, $ c_2 $ and $ c_3 $.
			In each local card $ U_l $ we get
			
			\begin{equation}\label{gal3}
				\| (\Delta^i)^{j_0/2} v \|_{L_i^{p_0} (X)} \Big|_{U_l}  \leq c_1
				\left( \|M_{j_0,l} \nabla^{j_0} v_{U_l}\|_{L^{p_0}(\mathbb{R}^n)} + \|N_{j_0,l} v_{U_l}\|_{L^{p_0}(\mathbb{R}^n)} \right) 
				\leq
			\end{equation}
			\[
			\tilde c_1\left( \|\nabla^{j_0} v_{U_l}\|_{L^{p_0}(\mathbb{R}^n)} + \| v_{U_l}\|_{L^{p_0}(\mathbb{R}^n)} \right) 
			\]
			if $ j_0 $ is even, and
			\begin{equation}\label{gal5}
				\left( \|A^i(\Delta^i)^{(j_0-1)/2}v\|_{L^{p_0}_{{i+1}}(X)} + \|(A^{i-1})^*(\Delta^i)^{(j_0-1)/2}v\|_{L^{p_0}_{{i-1}}(X)}\right)  \Big|_{U_l}  \leq
			\end{equation}
			\[
			c_2 \left( 	\|\tilde M_{j_0,l} \nabla^{j_0} v_{U_l}\|_{L^{p_0}(\mathbb{R}^n)} + \|\tilde K_{j_0,l} \nabla^{j_0} v_{U_l}\|_{L^{p_0}(\mathbb{R}^n)} + \|\tilde N_{j_0,l} v_{U_l}\|_{L^{p_0}(\mathbb{R}^n)} + \right. 
			\]
			\[
			\left. \|\tilde T_{j_0,l} v_{U_l}\|_{L^{p_0}(\mathbb{R}^n)}\right)  \leq \tilde c_2\left( \|\nabla^{j_0} v_{U_l}\|_{L^{p_0}(\mathbb{R}^n)} + \| v_{U_l}\|_{L^{p_0}(\mathbb{R}^n)} \right) 
			\]
			if $ j_0 $ is odd, with constants $c_1,\tilde c_1, c_2,\tilde c_2>0 $, where $ M_{j_0,l} $, $ N_{j_0,l} $, $ \tilde M_{j_0,l} $, $\tilde N_{j_0,l} $, $ \tilde K_{j_0,l} $ and $ \tilde T_{j_0,l} $ are some matrices with infinitely differentiable coefficients and $ v_{U_l} $ is  representation of $ v $ in $ U_l $ (see, for instance, \cite[Chapter 1]{Tarkhanov95}). Applying the Gagliardo-Nirenberg inequality (see \cite{Nir59} or \cite{Aubin82}) we have
			\begin{equation}\label{gal4}
				\|\nabla^{j_0} v_{U_l}\|_{L^{p_0}(\mathbb{R}^n)} \leq c\left( 
				\| \nabla^{m_0} v_{U_l} \|_{L^{r_0}(\mathbb{R}^n)}^a 
				\| v_{U_l} \|^{1-a}_{L^{q_0} (\mathbb{R}^n)}  +  \| v_{U_l} \|_{L^{2} (\mathbb{R}^n)}\right)
			\end{equation}
			with constant $ c>0 $,	where indexes are subordinate to (\ref{eq.L-G-N.cond}).
			On the other hand, G\aa{}rding's inequality imply 
			\begin{equation}\label{gal2}
				\| \nabla^{m_0} v_{U_l} \|_{L^{r_0}(\mathbb{R}^n)} \leq c \left( 
				\| \nabla_i ^{m_0} v \|_{L_i^{r_0}(X)} + \|\Pi^i_{m_0} v \|_{L_i^{r_0}(X)} \right)
			\end{equation}
			with positive constant $ c $.
			It follows from the ellipticity of $ \nabla_i ^{m_0} $ that dimension of the space $ \mathcal{ H}_m^i $ is finite, then $ \|\Pi^i_{m_0} v \|_{L_i^{r_0}(X)}\leq c \|\Pi^i_{m_0} v \|_{L_i^{2}(X)} $ with constant $ c>0 $, since all tow norms are equivalent on finite dimension space, and, moreover, there exist a constant $\tilde c>0 $ such that
			\begin{equation}\label{ogrproec}
				\|\Pi^i_{m_0} v \|_{L_i^{2}(X)}\leq c \| v \|_{L_i^{2}(X)}.
			\end{equation}
			Summing up by $ l $ the inequalities (\ref{gal3}) - (\ref{gal2}) and taking into account (\ref{ogrproec}) we receive (\ref{eq.L-G-N}).
		\end{proof}
		
		We need now Young's inequality: given any $N = 1, 2, \ldots$, it follows that
		\begin{equation}
			\label{eq.Young}
			\prod_{j=1}^N a_j \leq \sum_{j=1}^N \frac{a_j^{p_j}}{p_j}
		\end{equation}
		for all positive numbers $a_j$ and all numbers $p_j \geq 1$ satisfying
		$\displaystyle \sum_{j=1}^N 1/p_j = 1$.
		
		Next, the first three linear operators are continuous by the very definition of the function spaces.
		
		Let us prove the (\ref{eq8}). We start with $ s=1 $ and argue by the induction.
		The space $ B^{k+2,0,0}_{{i},\mathrm{vel}} (X_T) $ is continuously 
		embedded into the spaces   $C (I,V^i_{k+2s})$ and
		$L^2(I,V^i_{k+2s+1})$.
		
		First, we note that, by the Sobolev embedding theorem 
		for any
		$k, s \in {\mathbb Z}_+$ 
		satisfying
		\begin{equation}
			\label{eq.Sob.index}
			k - s > n/2,
		\end{equation}
		there exists a constant $c (k,s)$ depending on the parameters, such that
		\begin{equation*}
			\| u \|_{C^{s}_i (X)}
			\leq c (k,s)\, \| u \|_{H^{k}_i(X)}
		\end{equation*}
		for all $u \in H^{k}_i(X)$.
		
		Assume that $n\geq 3$ (for $n=2$ proof is similar with simpler arguments). Then by \eqref{eq.L-G-N} 
		\begin{equation} \label{eq.L-G-N.1}
			\|u\|_{L^{\frac{2n}{n-2}}_i (X)} \leq c\left(  \|\nabla_i u\|
			_{L^2_i (X)} + \| u \|_{L^{2}_i (X)}\right) 
		\end{equation}
		with the Gagliardo-Nirenberg constant $c$, 
		because
		\begin{equation*}
			p_0 = \frac{2n}{n-2}   \mbox{ if } a=1,\ j_0=0,\ m_0=1 \mbox{ and } r_0 = 2. 
		\end{equation*}
		If, in addition, $m'\geq n/2-1$ then by \eqref{eq.L-G-N} and the hypothesis of the lemma, 
		\begin{equation} \label{eq.L-G-N.n}
			\|u\|_{L^n_i (X)} \leq c \left( \|\nabla_i^{m'} u\|^{\frac{n-2}{2m'}}
			_{{L}^2_i (X)}
			\| u\|^{\frac{2m'-n+2}{2m'}}_{{L}^2_i (X)} + \| u \|_{L^{2}_i (X)}\right) 
		\end{equation}
		with the Gagliardo-Nirenberg constant $c$, 
		because
		\begin{equation*}
			\frac{1}{n} =\Big(\frac{1}{2}- \frac{m'}{n}\Big)a + \frac{1-a}{2}, 
			\,\, a = \frac{n-2}{2m'} \in (0,1). 
		\end{equation*}
		
		Using by \ref{eq9}, H\"older's inequality and G\aa{}rding's inequality we have
		\[
		\| \B (w,u) \|^2_{{L}^2_i(X)}
		\leq 
		\tilde{c}\left( \| w \|^2_{{L}_i^n(X)} \left( \| \nabla_i u \|^2 
		_{{L}_i^{\frac{2n}{n-2}}(X)}+ \| u \|^2 
		_{{L}_i^{2}(X)} \right) +\right. 
		\]
		\[
		\left.  \left( \| \nabla_i w \|^{2}_{{L}_i^{\frac{2n}{n-2}}(X)} + \| \nabla_i w \|^{2}_{{L}_i^{2}(X)}\right)  \| u \|^2
		_{{L}_i^n(X)}\right)
		\]
		with a positive constant $ \tilde{c} $, independent of $ u $ and $ w $.
		Then, with $m'=k+2>n/2-1$ in (\ref{eq.L-G-N.n}), 
		\[
		\| w \|^2_{{L}_i^n(X)}  \leq c_1
		\left( 
		\|\nabla_i^{k+2} w\|^{\frac{n-2}{2(k+2)}}_{{L}^2_i (X)}
		\| w\|^{\frac{2k-n+6}{2(k+2)}}_{{L}^2_i (X)}  + \| w \|
		_{{L}^2_i(X)}\right)^2  \leq c_2\| w \|^2_{{H}^{k+2}_i(X)},
		\]	
		\[
		\| \nabla_i u \|^2 
		_{{L}_i^{\frac{2n}{n-2}}(X)} \leq \tilde c_1
		\left( \|\nabla_i^2 u\|_{{L}_i^2 (X)} + \|\nabla_i u \|
		_{{L}^2_i(X)} \right)^2  \leq \tilde c_2\| u \|^2_{{H}^{2}_i(X)},
		\]	
		with positive constants $ {c_1} $, $ {c_2} $, $ \tilde{c_1} $, $ \tilde{c_2} $ independent of $ u $ and $ w $.	Then we have
		
		\begin{equation}\label{eq.B.cont.1}
			\| \B (w,u) \|^2_{{L}_i^2(X)}
			\leq 
		\end{equation}
		\[	c_1
		\left(\| w \|^2_{{H}^{k+2}_i(X)}\left( \| u \|^2_{{H}^{2}_i(X)} + \| u \| 
		_{{L}_i^{2}(X)} \right) \right.  + 
		\]
		\[
		\left. \| u \|^2_{{H}^{k+2}_i(X)}\left( \| w \|^2_{{H}^{2}_i(X)} + \| w \| _{{L}_i^{2}(X)} \right)\right) \leq
		\]
		\[
		\left. +\left( \|\nabla_i^2 w\|_{{L}^2_i (X)} + \|\nabla_i w \|
		_{{L}^2_i(X)} \right)^2 \| u \|^2_{{H}^{k+2}_i(X)}\right)\leq 
		\]
		\[	c_2 \| w \|^2
		_{{H}^{k+2}_i(X)} \| u \|^2_{{H}^{k+2}_i(X)},
		\]
		the positive constants $c_1$ 
		and $c_2$ being independent of $u$ and $w$, and so
		\begin{equation} \label{eq.s1.C.base}
			\|\B (w,u) \|^2_{C (I,{L}^2_i(X))}
			\leq c \| u \|^2_{C (I,{H}^{k+2}_i(X))} \| w \|^2_{C (I,{H}^{k+2}_i(X))}
		\end{equation}
		with constant $ c>0 $.
		
		For $n\geq 3$, by G\aa{}rding's and H\"older's inequalities we also have
		\begin{equation}
			\label{eq.B.cont.3}
			\| \nabla_i \B (w,u) \|^2_{{L}^2_i(X)} \leq 
			\tilde{c}
			\left(  \| w \|^2_{{L}^n _i(X)} \| \nabla_i^2 u \|^2
			_{{L}^{\frac{2n}{n-2}} _i(X)}
			+\right. 
		\end{equation}
		\[
		\| \nabla_i w \|^2_{{L}^n_i(X)} \| \nabla_i u \|^2
		_{{L}^{\frac{2n}{n-2}}_i(X)} + \| \nabla_i w \|^2_{{L}^{\frac{2n}{n-2}}_i(X)} \| \nabla_i u \|^2_{{L}^n_i(X)}
		+ 
		\]
		\[
		\|  w \|^2_{{L}^n_i(X)} \| \nabla_i u \|^2
		_{{L}^{\frac{2n}{n-2}}_i(X)} + \| \nabla_i w \|^2_{{L}^{\frac{2n}{n-2}}_i(X)} \|  u \|^2_{{L}^n_i(X)}
		+ 
		\]
		\[
		\|  w \|^2_{{L}^n_i(X)} \| u \|^2
		_{{L}^{\frac{2n}{n-2}}_i(X)} + \| w \|^2_{{L}^{\frac{2n}{n-2}}_i(X)} \|  u \|^2_{{L}^n_i(X)}
		+ 
		\]
		\[
		\left.\| \nabla_i^2 w \|^2_{{L}^{\frac{2n}{n-2}} _i(X)} \| u \|^2
		_{{L}^n_i(X)}\right) \leq c \left(  \| u \|^2_{{H}^{k+3}_i(X)} \| w \|^2_{{H}^{k+3}_i(X)}\right)
		\]
		with a constant $c$ independent of $u$ and $w$.
		On combining \eqref{eq.s1.C.base} and \eqref{eq.B.cont.3} we deduce that, for $n\geq 3$,
		\begin{eqnarray}
			\label{eq.s1.L.base}
			\| \B (w,u) \|^2_{L^2 (I,{H}^1_i(X))}
			\, \leq \,
			c\left(\| u \|^2_{{H}^{k+2}_i(X)} \| w \|^2_{{H}^{k+2}_i(X)} +\right. 	
			\\
			\left. + 
			\| u \|^2_{{H}^{k+3}_i(X)} \| w \|^2_{{H}^{k+3}_i(X)}\right).
			\nonumber
		\end{eqnarray}	
		
		Inequalities
		\eqref{eq.s1.C.base}, \eqref{eq.s1.L.base} provide that the operator 
		$\B (w,\cdot)$ maps 
		$ B^{2,0,0}_{{i},\mathrm{for}} (X_T)$ continuously
		to $  B^{0,0,0}_{{i},\mathrm{for}} (X_T)$ if $k>n/2-3$. 
		
		Next, for any $0 \leq k' \leq k_0$, $k_0=k$ or $k_0=k+1$,
		similarly to  \eqref{eq.B.cont.3}, using the H\"{o}lder and G\aa{}rding's inequality 
		with a number $q = q(k',l)>1$ we obtain 
		\begin{equation*}
			\| \nabla_i^{k'} \B (w,u) \|^2_{{L}_i^2(X)}
			\leq 
			\sum_{l=0}^{k'}
			\sum_{j=0}^{k'+1-l}
			\sum_{m=0}^{l}
			\Big(
			C^{j,m}_{k',l}\| \nabla_i^m w \|^2_{{L}_i^{\frac{2q}{q-1}}(X)} \| \nabla_i^{j} u \|^2
			_{{L}_i^{2q} (X)}
			+ 
		\end{equation*}
		\begin{equation}\label{eq25}
			\tilde C^{j,m}_{k',l} \| \nabla_i^{j} w \|^2_{{L}_i^{\frac{2q}{q-1}}(X)} \| \nabla_i^m u \|^2
			_{{L}_i^{2q}(X)}
			\Big)
		\end{equation}
		
		with positive coefficients $C^{j,m}_{k',l}$, $\tilde C^{j,m}_{k',l}$.
		
		If $0\leq k'\leq k_0$, then we take $q=q (k',0)=\frac{n}{n-2}$ and use \eqref{eq.Young}, 
		\eqref{eq.L-G-N.n} with $m'=k+2$, to obtain 
		\begin{equation} \label{eq.cont.k}
			\|\nabla_i^{j} u\|^2_{{L}_i^{2q} (X)} 
			\|w\|^2_{{L}_i^{\frac{2q}{q-1}}(X)} 
			=\|\nabla_i^{j} u\|^2_{{L}_i^{\frac{2n}{n-2}} (X)} 
			\| w\|^2_{{L}_i^{n}(X)} \leq 
		\end{equation}
		\[
		\tilde{c} \left( \|\nabla_i^{j+1} u\|_{{L}_i^{2} (X)} + \| \nabla_i^{j} u \|_{L^{2}_i (X)}\right)^2
		\left( 
		\|\nabla_i^{k+2} w\|^{\frac{n-2}{2(k+2)}}
		_{{L}_i^2 (X)}
		\| w\|^{\frac{2k-n+6}{2(k+2)}}_{{L}_i^2 (X)} + \| w \|_{L^{2}_i (X)}\right)^2  \leq
		\]
		\[
		c\| u\|^2_{{H}_i^{j + 1} (X)}
		\| w\|^{2}
		_{{H}_i^{k+2} (X)}
		\]
		for each $ 0\leq j \leq k'+1 $, with positive constants $ \tilde{c} $, $c$ independent on $u,w$.
		
		If $1\leq l \leq k'\leq k_0$ then we may apply \eqref{eq.L-G-N} 
		to each factor in the typical summand 
		\begin{equation*} 
			\|\nabla_i^m w\|^2_{{L}_i^{\frac{2q}{q-1}}(X)}
			\|\nabla_i^{j} u\|^2_{{L}_i^{2q} (X)} 
		\end{equation*}
		with $ 0\leq j \leq k'+1 - l $, $ 0\leq m \leq l $ and entries $q=q(k',l)$, $\alpha_j =  \alpha^{(l)}_j$,  satisfying 
		\begin{equation} \label{eq.cont.q}
			\left\{
			\begin{array}{lll}
				\frac{1}{2q} = \frac{j}{n} + \Big(\frac{1}{2} -\frac{k_0+2}{n}\Big) \alpha_1 + 
				\frac{1-\alpha_1}{2}, \\
				\frac{q-1}{2q} = \frac{m}{n} + \Big(\frac{1}{2} -\frac{k+2}{n} \Big) \alpha_2 + 
				\frac{1-\alpha_2}{2}, \\
				\frac{j}{k_0+2} \leq \alpha_1< 1, 
				\frac{m}{k+2} \leq \alpha_2< 1. 
			\end{array}
			\right.
		\end{equation}
		Relations \eqref{eq.cont.q} are actually equivalent to the following:
		\begin{equation} \label{eq.cont.alpha.1}
			\frac{j}{k_0+2} 
			\leq \alpha_1 =\frac{n}{k_0+2} \Big(\frac{1}{2}- \frac{1}{2q}+ \frac{j}{n}\Big)<1
		\end{equation}
		\begin{equation} \label{eq.cont.alpha.2}
			\frac{m}{k+2} \leq \alpha_2 =  \frac{n}{k+2} \Big(\frac{1}{2q} + 
			\frac{m}{n}\Big) <1 . 
		\end{equation}
		The lower bounds are always true if $q>1$ and so, 
		these inequalities are reduced to
		\begin{equation*}
			\frac{1}{2} + \frac{j - k_0 - 2}{n} 
			< \frac{1}{2q}< \frac{k+2-m}{n}, \, \, q>1.
		\end{equation*}
		The segment for $\frac{1}{2q}$ is  not empty 
		because 
		\[
		\frac{1}{2} + \frac{j - k_0 - 2}{n} 
		< \frac{k+2-m}{n}
		\]
		provided by the assumptions $k+3>n/2$, $0\leq k'\leq k_0$, and $j+m\leq k' + 1$. Moreover, as 
		\[
		\frac{1}{2} + \frac{j - k_0 - 2}{n} < \frac{1}{2}, \quad 
		\frac{k+2-m}{n} >0, 
		\]
		we see that there is a proper $q>1$ to achieve 
		\eqref{eq.cont.alpha.1}, \eqref{eq.cont.alpha.2}.
		
		Then, similarly to \eqref{eq.cont.k},
		\begin{equation} \label{eq.cont.kj} 
			\|\nabla_i^{j} u\|^2_{{L}_i^{2q} (X)} 
			\|\nabla_i^m w\|^2_{{L}_i^{\frac{2q}{q-1}}(X)} \leq
		\end{equation} 
		\[
		\tilde{c}
		\left( \|\nabla_i^{k_0+2} u\|^{\alpha_1}_{{L}_i^{2} (X)}\|u\|^{1-\alpha_1}_{{L}_i^{2}(X)} + \|  u \|_{L^{2}_i (X)}\right)^2 \cdot
		\]
		\[
		\left( \|\nabla_i^{k+2} w\|^{\alpha_2}_{{L}_i^{2} (X)}
		\|w\|^{1-\alpha_2}_{{L}_i^{2}(X)} + \|  w \|_{L^{2}_i (X)}\right)^2  \leq 
		\]
		\[
		c
		\| u\|^{2}_{{H}^{k_0+2}_i (X)} 
		\|w\|^{2}_{{H}^{k+2}_i (X)} 
		\]
		with positive constants $ \tilde c $, $c$  independent on $u,w$.
		
		Hence, \eqref{eq.cont.k}, \eqref{eq.cont.kj} yield 
		\begin{equation} \label{eq.B.cont.6}
			\|\B(w,u)\|^2_{C(I,{H}^k_i (X))}
			\leq
			c \| u\|^2_{C(I,{H}_i^{k+2} (X))} 
			\|w\|^2_{C(I,{H}_i^{k+2}(X))}, 
		\end{equation}
		
		\begin{equation} \label{eq.B.cont.7}
			\| \B(w,u)\|^2_{L^2(I,{H}^{k+1}_i (X))}
			\leq 
		\end{equation}
		\[
		\tilde	c \Big(\| u\|^2_{C(I,{H}_i^{k+2} (X))} 
		\|w\|^2_{L(I,{H}_i^{k+3}(X))} +
		\| w\|^2_{C(I,{H}_i^{k+2}(X))} 
		\|u\|^2_{L(I,{H}_i^{k+3}(X))} \Big) ,
		\]
		with positive constants $c$, $ \tilde	c $  independent on $u,w$. 
		
		Now \eqref{eq.B.cont.6}, \eqref{eq.B.cont.7} imply that the mapping 
		$\B (w,\cdot)$ maps $B^{k+2,0,0}_{{i},\mathrm{for}} (X_T)$ continuously 
		to $B^{k,0,0}_{{i},\mathrm{for}} (X_T)$ for any $k > n/2-3$ if $n\geq 3$  
		and bound 
		\eqref{eq.B.pos.bound} hold true for $s=1$.
		
		Next, we argue by the induction. Assume that for some $s'\geq 1$ 
		the mapping $\B (w,\cdot)$ maps $ B^{k+2,2(s'-1),s'-1} _{{i},\mathrm{for}}
		(X_T)$
		continuously to $ B^{k,2(s'-1),s'-1}_{{i},\mathrm{for}} (X_T)$ for any 
		$k > n/2-2s'-1$ and bound   
		\eqref{eq.B.pos.bound} holds true for $s=s'$. 
		Then the space $ B^{k+2,2s',s'}_{{i},\mathrm{for}} (X_T)$
		is continuously embedded to the space $ B^{k+4,2(s'-1),s'-1}_{{i},\mathrm{for}} 
		(X_T)$ and, 
		by the inductive assumption, 
		$\B (w,\cdot)$ maps  $B^{k+4,2(s'-1),s'-1}_{{i},\mathrm{for}}
		(X_T)$ continuously to  $B^{k+2,2(s'-1),s'-1}_{{i},\mathrm{for}}
		(X_T)$ for any $(k+2) > n/2-2s'-1$ 
		or, the same, $k> n/2-2(s'+1)-1$.  Moreover, bound  
		\eqref{eq.B.pos.bound} holds true for $s=s'$ and with $k+2$ instead of $k$.
		
		It is left to check the behaviour of the partial derivatives 
		$\partial^{s'}_t \nabla^{k_0}_i \B (w,u)$ with $k_0\leq k+1$. 
		By the very definition of space $ B^{k+2,2s',s'}_{{i},\mathrm{for}}(X_T)$,  
		the partial derivatives $\partial^{j'}_t u$, $\partial^{j'}_t w$ belong to  
		$C (I,{H}_i^{k+2+2(s'-j')}(X))$ and
		$L^2 (I, {H}_i^{k+3+2(s'-j')} (X))$. 
		
		By the Leibniz rule, 
		\begin{equation*} 
			\partial_t \B(w,u) =  \B(\partial_t w,u)  + \B(w,\partial_t u).
		\end{equation*}
		Then for any $0 \leq k' \leq k_0$, $k_0=k$ or $k_0=k+1$, $ 0\leq i'\leq s' $,
		similarly to  \eqref{eq.B.cont.3}
		\[
		\| \nabla_i^{k'} \partial_t^{i'} \B (w,u) \|^2_{{L}_i^2(X)}
		\leq 
		\]
		\[
		\sum_{j'=0}^{i'}
		\sum_{l=0}^{k'}
		\sum_{j=0}^{k'+1-l}
		\sum_{m=0}^{l}
		\Big(
		C^{j,m,j'}_{k',l,i'}\| \partial_t^{j'} \nabla_i^m  w \|^2_{{L}_i^{\frac{2q}{q-1}}(X)} \| \partial_t^{i'-j'} \nabla_i^{j} u \|^2
		_{{L}_i^{2q} (X)}
		+ 
		\]
		\begin{equation}\label{eq26}
			\tilde C^{j,m,j'}_{k',l,i'} \| \partial_t^{i'-j'} \nabla_i^{j} w \|^2_{{L}_i^{\frac{2q}{q-1}}(X)} \| \partial_t^{j'}\nabla_i^m u \|^2
			_{{L}_i^{2q}(X)}
			\Big)
		\end{equation}
		with positive coefficients $C^{j,m,j'}_{k',l,i'}$, $\tilde C^{j,m,j'}_{k',l,i'}$.
		
		Similarly to \eqref{eq.cont.k}, if $0\leq k'\leq k+1$ 
		then we take $q=q(k',0,s')=\frac{n}{n-2}$ and use \eqref{eq.Young}, \eqref{eq.L-G-N.n} 
		with $m'=2s'+k+2$, to obtain 
		\begin{equation} \label{eq.cont.ks}
			\|\partial_t^{s'} \nabla_i^{k'+1} u\|^2_{{L}_i^{2q} (X)} 
			\|  w\|^2_{{L}_i^{\frac{2q}{q-1}}(X)} 
			=\|\partial_t^{s'}  \nabla_i^{k'+1} u\|^2_{{L}_i^{\frac{2n}{n-2}} (X)} 
			\| w\|^2_{{L}_i^{n}(X)} \leq 
		\end{equation}
		\[
		\tilde c \left( \|\partial_t^{s'}  \nabla_i^{k'+2} u\|_{{L}^{2}_i (X)} 
		+ \|\partial_t^{s'} u\|_{{L}^{2}_i (X)} \right)^2 \cdot
		\]
		\[
		\left( \| \nabla_i^{2s'+k+2} w\|^{\frac{n-2}{2(2s'+k+2)}}
		_{{L}_i^2 (X)}\|w\|^{\frac{2(2s'+k)-n+6}{2(2s'+k+2)}}_{{L}_i^2 (X)}  + \|w\|_{{L}_i^2 (X)} \right)^2 \leq 
		\]
		\begin{equation*}
			c \|\partial_t^{s'} 
			u\|^2_{{H}^{k'+2}_i (X)} 
			\|  w\|^{2}
			_{{H}^{2s'+k+2}_i (X)}
		\end{equation*}
		with positive constants $\tilde c$, $ c$  independent on $u,w$.
		
		Again, similarly to \eqref{eq.cont.kj}, 
		if 
		\begin{equation*} 
			\begin{array}{ll}
				1\leq l \leq k'\leq k_0, &
				1\leq j' \leq s', \\
				0\leq j \leq k'+1 - l, &
				0\leq m \leq l,
			\end{array}
		\end{equation*}
		then we may apply \eqref{eq.L-G-N} 
		to each factor in the typical summand 
		\begin{equation*} 
			\|\partial_t ^{s'-j'} \nabla_i^{j} u\|^2_{{L}_i^{2q} (X)} 
			\|\partial_t ^{j'}  \nabla_i^m w\|^2_{{L}_i^{\frac{2q}{q-1}}(X)}
		\end{equation*}
		with entries satisfying 
		\begin{equation} \label{eq.cont.q.sj}
			\left\{
			\begin{array}{lll}
				\frac{1}{2q} = \frac{j}{n} + \Big(\frac{1}{2} -\frac{k_0+2+2j'}{n}\Big) \alpha_1 + 
				\frac{1-\alpha_1}{2}, \\
				\frac{q-1}{2q} = \frac{m}{n} + \Big(\frac{1}{2} -\frac{k+2+2(s'-j')}{n} \Big) \alpha_2 + 
				\frac{1-\alpha_2}{2}, \\
				\frac{j}{k_0+2+2j'} \leq \alpha_1< 1, 
				\frac{m}{k+2+2(s'-j')} \leq \alpha_2< 1. 
			\end{array}
			\right.
		\end{equation}
		Relations \eqref{eq.cont.q.sj} are actually equivalent to the following:
		\begin{equation} \label{eq.cont.alpha.1.sj}
			\frac{j}{k_0+2+2j'} 
			\leq \alpha_1 =\frac{n}{k_0+2+2j'} \Big(\frac{1}{2}- \frac{1}{2q}+ \frac{j}{n}\Big)<1,
		\end{equation}
		\begin{equation} \label{eq.cont.alpha.2.sj}
			\frac{m}{k+2+2(s'-j')} \leq \alpha_2 =  \frac{n}{k+2+2(s'-j')} \Big(\frac{1}{2q} + 
			\frac{m}{n}\Big) <1 . 
		\end{equation}
		The lower bounds are always true if $q>1$ and so, 
		these inequalities are reduced to
		\begin{equation*}
			\frac{1}{2} + \frac{j- k_0-2-2j'}{n} 
			< \frac{1}{2q}< \frac{k+2+2(s'-j')-m}{n}, \, \, q>1.
		\end{equation*}
		The segment for $\frac{1}{2q}$ is  not empty 
		because 
		\begin{equation*}
			\frac{1}{2} + \frac{j- k_0-2-2j'}{n} 
			< \frac{k+2+2(s'-j')-m}{n},
		\end{equation*}
		provided by the assumptions $k>n/2-2(s'+1)-1$, $0\leq k'\leq k_0$ and $j+m\leq k' + 1$. Moreover, as 
		\begin{equation*}
			\frac{1}{2} + \frac{j- k_0-2-2j'}{n} < \frac{1}{2}, \,\, \frac{k+2+2(s'-j')-m}{n}>0,
		\end{equation*}
		we see that there is a proper $q>1$ to achieve 
		\eqref{eq.cont.alpha.1.sj}, \eqref{eq.cont.alpha.2.sj}.
		
		Then, similarly to \eqref{eq.cont.k},
		\begin{equation} \label{eq.cont.kj.s} 
			\|\partial_t ^{s'-j'} \nabla_i^{j} u\|^2_{{L}_i^{2q} (X)} 
			\|\partial_t ^{j'}  \nabla_i^m w\|^2_{{L}_i^{\frac{2q}{q-1}}(X)} \leq
		\end{equation} 
		\[
		\tilde{c}
		\left( \|\partial_t^{s'-j'} \nabla_i^{k_0+2+2j'} u\|^{\alpha_1}_{{L}_i^{2} (X)} 
		\|\partial_t^{s'-j'} u\|^{1-\alpha_1}_{{L}^{2}_i(X)}
		+ \|\partial_t^{s'-j'} u\|_{{L}^{2}_i(X)}\right)^2 \cdot
		\]
		\[
		\left( \|\partial_t^{j'} \nabla_i^{k+2+2(s'-j')} w\|^{\alpha_2}_{{L}_i^{2} (X)}
		\|\partial_t^{j'} w\|^{1-\alpha_2}_{{L}^{2}_i(X)}
		+ \|\partial_t^{j'} w\|_{{L}_i^{2}(X)}\right)^2  \leq 
		\]
		\[
		c
		\| \partial_t^{s'-j'}  u\|^2_{{H}_i^{k_0+2+2j'} (X)} 
		\|\partial_t^{j'} w\|^2_{{H}^{k+2+2(s'-j')} (X)} 
		\]
		with positive constants $ \tilde{c} $, $c$  independent on $u,w$.
		
		Hence, \eqref{eq.cont.ks}, \eqref{eq.cont.kj.s} yield 
		\begin{equation} \label{eq.B.cont.6.sj}
			\|\partial_t^{s'} \B(w,u)\|^2_{C(I,{H}^k_i (X))}
			\leq 
		\end{equation}
		\[
		\sum_{j'=0}^{s'}\left( c_{j'}\| \partial_t^{s'-j'} u\|^2_{C(I,{H}_i^{k+2+2j'} (X))} 
		\|\partial_t^{j'} w\|^2_{C(I,{H}_i^{k+2+2 (s'-j')}(X))} +\right. 
		\]
		\[
		\left. \tilde c_{j'}\|\partial_t^{s'-j'} w\|^2_{C(I,{H}_i^{k+2+2j'} (X))} 
		\|\partial_t^{j'} u\|^2_{C(I,{H}_i^{k+2 +2 (s'-j')}(X))}  \right) , 
		\]
		\begin{equation} \label{eq.B.cont.7.sj}
			\| \partial_t^{s'} \B(w,u)\|^2_{L^2(I,{H}_i^{k+1} (X))}
			\leq 
		\end{equation}
		\[
		\sum_{j'=0}^{s'}\left( c_{j'}\| \partial_t^{s'-j'} u\|^2_{L^2(I,{H}_i^{k+3+2j'} (X))} 
		\|\partial_t^{j'} w\|^2_{L^2(I,{H}_i^{k+2+2 (s'-j')}(X))} +\right. 
		\]
		\[
		\left. \tilde c_{j'}\|\partial_t^{s'-j'} w\|^2_{L^2(I,{H}_i^{k+2+2j'} (X))} 
		\|\partial_t^{j'} u\|^2_{L^2(I,{H}_i^{k+ 3 +2 (s'-j')}(X))}  \right) , 
		\]
		with positive constants $c_{j'}$ and $ \tilde c_{j'} $, $ 0\leq j'\leq s' $, independent on $u$, $v$. 
		
		Now \eqref{eq.B.cont.6.sj}, \eqref{eq.B.cont.7.sj} imply that the mapping 
		$\B (w,\cdot)$ maps $ B^{k+2,2s',s'}_{{i},\mathrm{for}} (X_T)$ 
		continuously to $B^{k,2s',s'}_{{i},\mathrm{for}} (X_T)$ if $n\geq 3$. 
		Moreover, by \eqref{eq.Young}, bound  
		\eqref{eq.B.pos.bound} holds true for $s=s'+1$.
		This finishes the proof of inequality \eqref{eq.B.pos.bound} 
		and the continuity of operator $\B (w,\cdot) :
		B^{k+2,2(s-1),s-1}_{{i},\mathrm{vel}} (X_T)
		\to  B^{k,2(s-1),s-1}_{{i},\mathrm{for}} (X_T)$,
		for $n\geq 3$ and for  all   $k \in {\mathbb Z}_+$ and  $s \in \mathbb N$ satisfying 
		$2s+k>n/2-1$. 
		
		The boundedness of the operator
		$
		\B (w,\cdot) :
		B^{k,2s,s}_{{i},\mathrm{vel}} (X_T)
		\to
		B^{k,2(s-1),s-1}_{{i},\mathrm{for}} (X_T)
		$
		now follows from the definition of the spaces.
	\end{proof}
	
	Let us introduce now the Helmholtz type projection 
	$\mathrm {P}^i$ from $ B^{k,2(s-1),s-1}_{i,\mathrm {for}} (X_T)$ to the kernel of 
	operator $ (A^{i-1})^* $.
	
	\begin{lemma}\label{proector} 
		Let $ s $, $ k \in \mathbb{Z}_+ $. 
		For each $i$ the pseudo-differential operator $ \mathrm {P}^i = (A ^i)^* A^i \varphi^i + \Pi^i$ on X induce continuous map
		\begin{equation}\label{cont.proetor}
			\mathrm {P}^i: B^{k,2(s-1),s-1}_{i,\mathrm {for}} (X_T)\to B^{k,2(s-1),s-1}_{i, \mathrm {vel}} (X_T),
		\end{equation}
		such that
		\begin{equation*}
			\mathrm {P}^i\circ \mathrm {P}^i u = \mathrm {P}^i u,\quad
			(\mathrm {P}^i u, v)_{L^2_i(X)} = (u, \mathrm {P}^i v)_{L^2_i(X)},\quad
			(\mathrm {P}^i u, (I-\mathrm {P}^i) u)_{L^2_i(X)} = 0
		\end{equation*}
		for all $ u,v\in B^{k,2(s-1),s-1}_{i,\mathrm {for}} $.
	\end{lemma}
	\begin{proof}
		Indeed, as $A^{i+1} \circ A^i =0$, 
		\[
		\mathrm {P}^i\circ \mathrm {P}^i u = ((A ^i)^* A^i \varphi^i  + \Pi^i )\circ ((A ^i)^* A^i \varphi^i + \Pi^i ) u = 
		\]
		\[
		((A ^i)^* A^i \varphi^i (A ^i)^* A^i\varphi^i  + \Pi^i )u = ((A ^i)^* A^i \Delta^i\varphi^i  + \Pi^i )u.
		\]
		It follows from Theorem \ref{parametrix.laplas} that
		\[
		\mathrm {P}^i = I - A ^{i-1} (A^{i-1})^* \varphi^i,
		\]
		and then
		\[
		(\mathrm {P}^i u, v)_{L^2_i(X)} = (\mathrm {P}^i  u, \mathrm {P}^i v + A ^{i-1} (A^{i-1})^* \varphi^i v)_{L^2_i(X)} = (\mathrm {P}^i  u, \mathrm {P}^i v)_{L^2_i(X)} = 
		\]
		\[
		(u - A ^{i-1} (A^{i-1})^* \varphi^i u, \mathrm {P}^i v)_{L^2_i(X)} = ( u, \mathrm {P}^i v)_{L^2_i(X)},
		\]
		because  $(A^{i-1})^* \mathrm {P}^i = 0$. On the other hand,
		\[
		(\mathrm {P}^i u, (I-\mathrm {P}^i) u)_{L^2_i(X)} = (\mathrm {P}^i u, u)_{L^2_i(X)} - (\mathrm {P}^i u, \mathrm {P}^i u)_{L^2_i(X)} = 0.
		\]
		
		Finally, the continuity of $ \mathrm {P}^i: B^{k,2(s-1),s-1}_{i,\mathrm {for}} (X_T)\to B^{k,2(s-1),
			s-1}_{i, \mathrm {vel}} (X_T) $ follows from Theorem \ref{l.NS.cont.0} and 
		the commutative relations $ \mathrm {P}^i\partial_t^j = \partial_t^j \mathrm {P}^i$ with 
		$ j\leq s-1  $.
	\end{proof}
	
	\section{An open mapping theorem}
	
	Consider now the linearisation of problem (\ref{eq3}): given sufficiently regular section $f,w$ of the induced bundle $E_i (t)$
	and section $u_0$ of the bundle $E_i$, find sufficiently regular  sections
	$u$ and $p $ of the induced bundles $E_i (t)$ and $E_{i-1} (t)$
	which satisfy
	\begin{equation}\label{eq.NS.lin}
		\begin{cases}
			\partial_t v + \mu \Delta^{i} v + \B (w,v) + A^{i-1} p = f& \text{in } X_T,\\
			(A^{i-1})^* v =0, & \text{in } X_T,\\
			(A^{i-2})^* p =0 & \text{in } X_T,\\
			v(x,0) = v_0& \text{in } X,
		\end{cases}
	\end{equation}
	
	Now we want to show that \eqref{eq.NS.lin} has one and only one solution in the spaces, 
	introduced in first paragraph. We start with the following simple 
	corollary of the Hodge Theorem \ref{parametrix.laplas}.
	
	\begin{corollary}\label{p.nabla.Bochner} 
		Let $F\in B^{k,2(s-1),s-1}_{i, \mathrm {for}} (X_T)$ 
		satisfy $\mathrm {P}^i F=0$ in $X_T$.  
		Then there is a unique section 
		$p \in B^{k+1,2(s-1),s-1}_{i-1,\mathrm {pre}}(X_T)$
		such that (\ref{eq.p.ort.0}) holds
		and 
		\begin{equation}\label{eq.nabla.Bochner}
			A^{i-1} p = F \mbox{ in } X \times [0,T].
		\end{equation}
	\end{corollary}
	\begin{proof}
		Under the hypotheses of Theorem, the section
		\[
		p = (A^{i-1})^*\varphi^i F
		\]
		is a solution of (\ref{eq.nabla.Bochner}). Indeed,
		\[
		A^{i-1}p = A^{i-1}(A^{i-1})^*\varphi^i F,
		\]
		but it follows from Theorem \ref{parametrix.laplas} and Lemma \ref{proector} that
		\[
		F = A ^{i-1} (A^{i-1})^* \varphi^i F + \mathrm {P}^i F = A ^{i-1} (A^{i-1})^* \varphi^i F,
		\]
		because $  \mathrm {P}^i F = 0 $, then $ A^{i-1}p = F $ and $ (A^{i-2})^* p =0 $ by the construction of the solution. Moreover
		\begin{equation*}
			( p,h)_{L^2_{i-1}(X)} = ((A^{i-1})^*\varphi^i F,h)_{L^2_{i-1}(X)} = (\varphi^i F, A^{i-1} h)_{L^2_{i}(X)} = 0
		\end{equation*}
		for all $h \in \mathcal{H}^{i-1} $, hence $ p\in B^{k+1,2(s-1),s-1}_{i-1,\mathrm {pre}}(X_T) $.
		
		Let now $ p_1, p_2\in B^{k+1,2(s-1),s-1}_{i-1,\mathrm {pre}}(X_T) $ are two solutions of (\ref{eq.nabla.Bochner}). Then $ p = p_1 - p_2 $ is solution too and $ A^{i-1}p =0 $. Since $ p\in B^{k+1,2(s-1),s-1}_{i-1,\mathrm{pre}}(X_T) $, then
		$ (A^{i-2})^*p=0 $ and $ p $ actually belong to the harmonic space 
		$ \mathcal{H}^{i-1} $. 
		However, as
		\[
		( p,h)_{L^2_{i-1}(X)} =0
		\]
		for all $h \in \mathcal{H}^{i-1} $, then $ p\equiv 0 $.
	\end{proof}
	
	Let us recall the following useful lemma by J.-L. Lions.
	\begin{lemma}
		\label{l.Lions}
		Let $V$, $H$ and $V'$ be Hilbert spaces such that $V'$ is the dual to $V$ and the embeddings
		$
		V \subset H \subset V'
		$
		are continuous and everywhere dense.
		If
		$u \in L^2 (I,V)$ and
		$\partial_t u \in L^2 (I,V')$
		then
		\begin{equation}
			\label{eq.dt}
			\frac{d}{dt} \| u (\cdot, t) \|^2_{H}
			= 2\, \langle \partial_t u, u \rangle
		\end{equation}
		and $u$ is equal almost everywhere to a continuous mapping from $[0,T]$ to $H$.
	\end{lemma}
	
	\begin{proof}
		See \cite[Ch.~III, \S~1, Lemma~1.2]{Temam79}.
	\end{proof}
	
	\begin{theorem}
		\label{t.exist.NS.lin.weak}
		Let $n\geq 2$ and suppose that $w \in
		L^2 (I,V_i^1) \cap 
		L^{2}(I,{L}_{{i}}^{\infty} (X)) \cap 
		L^{\infty}(I,{L}_{{i}}^{n} (X)) $. 
		Given any pair $(f,u_0) \in L^2 (I,(V^1_i)') \times V^0_i$, there is a unique section
		$u \in C (I,V_i^0) \cap L^2 (I,V_i^1)$ with
		$\partial_t u \in L^2 (I,(V_i^1)')$,
		satisfying
		\begin{equation}
			\label{eq.NS.lin.weak}
			\left\{
			\begin{array}{rcl}
				\displaystyle
				\frac{d}{dt} (u,v)_{{L}_{{i}}^2(X)}
				+ \mu (A^i u, A^i v)
				_{{L}_{{i+1}}^2(X)}
				& =
				& \langle f - \B (w,u), v \rangle,
				\\
				u (\cdot,0)
				& =
				& u_0
			\end{array}
			\right.
		\end{equation}
		for all $v \in V_i^k$ with $ k\geq n/2 $.
	\end{theorem}
	
	\begin{proof}
		The proof is based on the standard method 
		by Faedo-Galerkin, energy estimates and Gagliardo-Nirenberg inequality, see, for instance, \cite{LiMa72}, \cite{Temam79} or \cite{Ladyz}.
	\end{proof}

	\begin{theorem}
		\label{t.exist.NS.lin.strong}
		Let $n \geq 2$, 
		$s \in \mathbb N$, $k \in {\mathbb Z}_+$, $2s+k> n/2$,  
		and
		$w \in B^{k,2s,s}_{i,\mathrm{vel}} (X_T)$.
		Then \eqref{eq.NS.lin} induces a bijective continuous linear mapping
		\begin{equation}
			\label{eq.map.Aw}
			\mathcal{A}_{w,i} :
			B^{k,2s,s}_{i,\mathrm{vel}} (X_T) \times 
			B_{i-1,\mathrm{pre}}^{k+1,2(s-1),s-1} (X_T)
			\to     B_{i,\mathrm{for}}^{k,2(s-1),s-1} (X_T) \times V_i^{2s+k}.
		\end{equation}
		which admits a continuous inverse $\mathcal{A}^{-1}_{w,i}$.
	\end{theorem}
	
	Note, if $2s+k >n/2-1$, then the space $ B^{k,2s,s}_{i,\mathrm{vel}} (X_T) $ is continuously embedded into the space $ L^{2}(I,{L}_{{i}}^{\infty} (X)) \cap 
	L^{\infty}(I,{L}_{{i}}^{n} (X)) $. Indeed, it follows from \eqref{eq.Sob.index} that $ L^{2}(I,{H}_{i}^{k+2s+1} (X)) \hookrightarrow  L^{2}(I,{L}_{{i}}^{\infty} (X)) $. On the other hand, \eqref{eq.L-G-N.n} gives $ C(I,{H}_{i}^{k+2s} (X))\hookrightarrow L^{\infty}(I,{L}_{{i}}^{n} (X)) $, then
	\begin{equation}\label{embb}
		B^{k,2s,s}_{i,\mathrm{vel}} (X_T) \hookrightarrow L^{2}(I,{L}_{{i}}^{\infty} (X)) \cap 
		L^{\infty}(I,{L}_{{i}}^{n} (X))
	\end{equation}
	with $2s+k >n/2-1$.
	
	\begin{proof} Again, we may follow the standard scheme of solving 
		nonlinear parabolic equations, 
		see, for instance, \cite{Lion69}, \cite{LiMa72}, \cite{Temam79}, 
		\cite{LadSoUr67} \cite{Ladyz}.
		
		We begin with a simple lemma.
		\begin{lemma} If    $s \in \mathbb N$, $k \in {\mathbb Z}_+$, $2s+k> n/2-1$ 
			then \eqref{eq.map.Aw} is an injective continuous linear mapping.
		\end{lemma}
		
		\begin{proof} 
			Indeed, the continuity of $\mathcal{A}_{w,i}$ follows from 
			Theorem \ref{l.NS.cont.0}.
			Let
			$$
			\begin{array}{rcl}
				(u,p)
				& \in
				&  B^{k,2s,s}_{i,\mathrm{vel}} (X_T) \times 
				B_{i-1,\mathrm{pre}}^{k+1,2(s-1),s-1} (X_T),
				\\
				\mathcal{A}_{w,i} (u,p) = (f,u_0)
				& \in
				& B_{i,\mathrm{for}}^{k,2(s-1),s-1} (X_T) \times V_i^{2s+k}.
			\end{array}
			$$
			As $ (A^{i-1})^* u = 0 $ in $X_T$ we have
			\begin{equation}
				\label{eq.by.parts.1}
				(\Delta^i u,u)_{{L}_i^2(X)}
				= \| A^i u \|^2_{{L}_{i+1}^2(X)}
			\end{equation}
			and
			\begin{equation}
				\label{eq.Duu.zero.1}
				(A^{i-1} p,u)_{{L}_i^2(X)}
				= (p, (A^{i-1})^* u)_{L_{i-1}^2(X)}
				= 0.
			\end{equation}
			As $2s+k+1>n/2$, by the Sobolev embedding theorems, see \eqref{eq.Sob.index}, 
			the space $B^{k,2s,s}_{i,\mathrm{vel}} (X_T)$
			is continuously embedded into 
			$L^2 (I,{L}_i^\infty(X))$. 
			Then formulas
			\eqref{eq.dt},
			\eqref{eq.by.parts.1} and
			\eqref{eq.Duu.zero.1}
			readily imply that $u$ is a weak solution to \eqref{eq.NS.lin} granted 
			by Theorem \ref{t.exist.NS.lin.weak}, i.e., \eqref{eq.NS.lin.weak}
			is fulfilled.
			
		\end{proof}
		
		Let us continue with the proof of the surjectivity. 
		
		\begin{lemma}
			\label{l.Galerkin} 
			Let $s\in \mathbb N$, $k \in {\mathbb Z}_+$, $2s+k>n/2$ and  $w\in 
			B^{k,2s,s}_{i,\mathrm{vel}} (X_T)$.  
			Then for each $(f, u_0)\in  B^{k,2(s-1),s-1}_{i,\mathrm{for}} (X_T) 
			\times V^{2s+k}_i $ there is a solution $u\in B^{k,2s,s}_{i,\mathrm{vel}}
			(X_T)$ to \eqref{eq.NS.lin.weak}.
		\end{lemma}
		
		\begin{proof}
			Let
			$(f,u_0)$ be an arbitrary data in $ B
			^{k,2(s-1),s-1}_{i,\mathrm{for}} (X) \times 
			V^{2s+k}_i$
			and let 
			$\{u_m\}$ be the sequence of the corresponding Faedo-Galerkin approximations, 
			namely, 
			\begin{equation*}
				u_m = \sum_{j=1}^M g^{(m)}_{j} (t) b_j (x),
			\end{equation*}
			where the system $\{ b_j \}_{j \in {\mathbb N}}$ 
			is a ${L}_i^2 (X)$-orthogonal basis in $V^0_i$ and the functions $g^{(m)}_{j} $ satisfy the following relations 
			\begin{equation} \label{eq.ODU.lin.0}
				(\partial_\tau u_m , b_j)_{{L}^2_i (X)} + \mu 
				(A^i u_m , A^i b_j )_{{L}^2_{i+1} (X)}  + 
				(\B(w, u_m) , b_j )_{{L}^2_i (X)} =
				\langle f, b_j\rangle ,
			\end{equation}
			\begin{equation*}
				u_m(x,0) = u_{0,m} (x)
			\end{equation*}
			for all $0\leq j \leq m$ with the initial datum $u_{0,m}$ from 
			the linear span ${\mathcal L} (\{b_j\}_{j=1}^m)$ such that the sequence 
			$\{u_{0,m}\}$ converges to $u_0$ in $H_0$.  
			For instance, as $\{u_{0,m}\}$ we may take  the orthogonal projection 
			onto the linear span ${\mathcal L} (\{b_j\}_{j=1}^m)$.

			The scalar functions 
			$F_{j} (t)= \langle f (\cdot,t),  b_{j}\rangle$ 
			belong to $C^{s-1} [0,T] \cap H^s [0,T]$, 
			and the components
			\begin{equation} \label{eq.C}
				\mathfrak{C}^{(m)}_{k,j} (t)= \mu 
				(A^i b_k , A^i b_j )_{{L}^2_{i+1} (X)}+
				(\B(w(\cdot,t), b_k) , b_j )_{{L}^2_i (X)} =
			\end{equation}
			\begin{equation*}
				\mu 
				(A^i b_k , A^i b_j )_{L^2_{i+1} (X)} +
				(M_{i,1} ( A^{i}w, b_k), b_j)_{{L^2_i}(X)}  + (M_{i,1} ( A^{i}b_k, w), b_j)_{{L^2_i}(X)} 
			\end{equation*}
			belong to 
			$C^{s} [0,T] \cap H^{s+1} [0,T]$. 
			Since $w \in B^{k,2s,s}_{\mathrm{i, vel}} (X_T)$, 
			formula \eqref{eq.C} means that  
			the entries of the matrix $\exp{\Big(\int_0^t {\mathfrak C}^{(m)}(\tau) d\tau \Big)}$
			belong actually to $C^{s+1} [0,T]\cap H^{s+2} [0,T]$ 
			and then the components 
			of the vector $g^{(m)}$ belong to $C^{s} [0,T] \cap H^{s+1} [0,T]$. 
			
			Again we assume that $n\geq 3$ (for $n=2$ arguments are similar).
			Let $s=1$ and $k\in {\mathbb Z}_+$ satisfying $k>n/2-2$.
			If we multiply the equation corresponding to index $j$ in 
			\eqref{eq.ODU.lin.0} by $ \frac{d g^{(m)}_{j}}{d\tau}$ then, after the summation with 
			respect to $j$, we obtain for all $\tau \in [0,T]$:
			\begin{equation} \label{eq.um.bound.add}
				\|\partial_ \tau u_m\|^2_{{L}^2 (X)} + 
				\frac{\mu }{2} \frac{d}{\partial \tau }\|A^i u_m \|^2_{{ L}^2 (X)} 
				= 
			\end{equation} 
			\[
			(f , \partial_\tau u_m )_{{L}^2_i (X) } - (M_{i,1} ( A^{i}u_m, w), \partial_\tau u_m)_{{L^2_i}(X)} -
			(M_{i,1} ( A^{i}w, u_m), \partial_\tau u_m)_{{L^2_i}(X)}.
			\]

			It follows from (\ref{eq9}) and from the H\"older inequality with $q_1=\infty$, $q_2=2$, $q_3=2$ and $p_1=\frac{2n}{n-2}$, $p_2=n$, $p_3=2$ that
			\begin{equation}\label{ineq1}
				\left| (M_{i,1} ( A^{i}u_m, w), \partial_\tau u_m)_{{L^2_i}(X)} + (M_{i,1} ( A^{i}w, u_m), \partial_\tau u_m)_{{L^2_i}(X)} \right|  \leq
			\end{equation}
			\[
			c\left( \|A^i u_m \|_{{ L}_{i+1}^2 (X)} \|w \|_{{ L}_i^\infty (X)}+	\| u_m \|_{{ L}_i^{\frac{2n}{n-2}} (X)}\|A^i w \|_{{ L}_{i+1}^n (X)} \right) \|\partial_ \tau u_m\|_{{ L}_i^2 (X)}\leq
			\]
			\[
			c_1\left(  \|A^i u_m \|^2_{{ L}_{i+1}^2 (X)} \|w \|^2_{{ L}_i^\infty (X)}+	\| u_m \|^2_{{ L}_i^{\frac{2n}{n-2}} (X)}\|A^i w \|^2_{{ L}_{i+1}^n (X)} \right)  + \frac{1}{2}\|\partial_ \tau u_m\|^2_{{ L}_i^2 (X)}
			\]
			with positive constants $ c $ and $ c_1 $. The last estimation is due to the Young's inequality \eqref{eq.Young}. By the inequality (\ref{eq.L-G-N.n}) with $ m'=k+1 $ and Lemma (\ref{equivalent}) there are constants $ c,c_1,c_2>0 $ such that 
			\begin{equation}\label{eneq11}
				\|A^iw\|^2_{ L_{i+1}^n (X)} \leq c_1 \left( \|\nabla_i w\|^2_{ L_i^n (X)} + \| w\|^2_{ L_i^2 (X)}\right)\leq
			\end{equation}
			\[
			c_2 \left(  \|\nabla_i^{k+2} w\|^{\frac{n-2}{k+1}}
			_{{L}^2_i (X)}
			\| w\|^{\frac{2(k+1)-n+2}{k+1}}_{{L}^2_i (X)} + \| \nabla_i w \|^2_{L^{2}_i (X)} + \| w\|^2_{ L_i^2 (X)}\right)\leq c\| w\|^2_{ H_i^{k+2} (X)}.
			\]
			On the other hand \eqref{eq.L-G-N.1} gives
			
			\begin{equation}\label{eneq12}
				\| u_m \|^2_{{ L}_i^{\frac{2n}{n-2}} (X)} \leq c \left( \|\nabla_i u_m\|_{ L_i^2 (X)} + \| u_m\|_{ L_i^2 (X)}\right)^2\leq
			\end{equation}
			
			\[
			c_1 \left( \|A^i u_m\|^2_{ L_{i+1}^2 (X)} + \| u_m\|^2_{ L_i^2 (X)}\right),
			\]
			with positive constants $ c $, $ c_1 $. The last estimate is consequence of G\aa{}rding's inequality
			\[
			\| u_m\|_{ H_i^1 (X)}\leq c\left( \| A^i u_m\|_{  L_{i+1}^2 (X)} + \| (A^{i-1})^* u_m\|_{  L_{i-1}^2 (X)} +\| u_m\|_{  L_i^2 (X)}\right)
			\]
			with a constant $ c>0 $,
			however $  (A^{i-1})^* u_m = 0$.
			
			After the integration of (\ref{eq.um.bound.add}) with respect to $\tau \in I_t = [0,t]$ we arrive at the following:
			\begin{equation} \label{eq.um.bound.1}
				2\|\partial_ \tau u_m\|^2_{L^2(I_t,{ L}_i^2 (X))} + 
				\mu \|A^i u_m \|^2_{{ L}_{i+1}^2 (X)} 
				\leq  \mu\|A^i {u_0}_m\|^2_{{ L}_i^2 (X)} + 
			\end{equation}
			\[
			2\|f\|_{L^2(I_t,{ L}_i^2 (X))}\|\partial_ \tau u_m\|_{L^2(I_t,{ L}_i^2 (X))}+ \frac{1}{2}\|\partial_ \tau u_m\|^2_{L^2(I_t,{ L}_i^2 (X))} +
			\]
			\[
			c\int_0^t \left(  \|  w \|^2_{{ L}_i^\infty (X)} \| A^i u_m \|^2
			_{{ L}_i^2 (X)} + \| u_m \|^2_{{ L}_i^{\frac{2n}{n-2}} (X)}
			\|A^i w \|^2_{{ L}_i^n (X)} 
			\right)  \, d\tau \leq 
			\]
			\[
			\mu\|A^i {u_0}_m\|^2_{L^2(I_t,{ L}_i^2 (X))} + 
			4 \|f\|^2_{L^2 (I,{L}^2_i (X)) } + \|\partial_ \tau u_m\|^2_{L^2(I_t,{ L}_i^2 (X))}+
			\]
			\[
			c_1 \int_0^t \|  w \|^2_{{ H}^{k+2}_i (X)} \left( 
			\| A^i u_m \|^2_{{ L}^2_{i+1} (X)} + \| u_m \|^2_{{ L}^2_i (X)}\right)   d\tau  ,
			\]
			with positive constants $c$, $c_1$ independent on $w$ and $m$, 
			the last bound being a consequence of the Sobolev Embedding Theorems and 
			inequalities \eqref{eneq11}, \eqref{eneq12} and \eqref{ineq1}.
			
			\begin{lemma}
				\label{l.OM.bound}
				Suppose $k \in {\mathbb Z}_+$ satisfying $k>n/2-2$.
				If $(f,u_0) \in B^{k,0,0}_{i,\mathrm{for}} (X_T) \times V^{k+2}_i$ then
				\begin{equation}\label{povish}
					\| \nabla^{k'}_i u_m \|^2_{C (I,{L}^2_i   )}
					+ \mu\, \| \nabla^{k'+1}_i u_m \|^2_{L^2 (I,{L}^2_i   )}
					\leq
					c_{k'} (\mu,w,f,u_0)
				\end{equation}
				for any $0 \leq k' \leq k+2$, the constants
				$c_{k'} (\mu,w,f,u_0) > 0$
				depending on
				$k'$ and
				$\mu$
				and the norms
				$\| w \|_{B^{k,2,1}_\mathrm{i, vel} (X_T)}$,
				$\| f \|_{B^{k,0,0}_{i,\mathrm{for}} (X_T)}$,
				$\| u_0 \|_{V^{k+2}_i}$
				but not on $m$.
			\end{lemma}
			
			\begin{proof}
				We argue by induction. Let $k=0$ and  $k'=0$, 
				substituting $ u_m $ into \eqref{eq.NS.lin.weak} instead of $ v $ and $ u $ we have
				\begin{eqnarray}
					\label{eq67}
					\frac{d}{dt} \| u_m \|^2_{{L}_i^2(X)   }
					+ 2 \mu \| \nabla_i u_m \|^2_{{L}^2_i(X)    } = 
					2 \left( (\B (w,u_m) - f), u_m  \right) _{{L}^2_i(X) }
				\end{eqnarray}
				for all $t \in [0,T]$.
				Using by H\"older inequality, we get
				\begin{equation}
					\label{eq68}
					2\left|  \left( f, u_m  \right) _{{L}^2_i(X) }\right| 
					\leq
					\frac{2}{\mu} \| f \|^2_{{L}^2_i(X)  }
					+ \frac{\mu}{2} \| u_m \|^2 _{{L}^2_i(X)   } ,
				\end{equation}
				and similarly
				\begin{equation}
					\label{eq69}
					2 \left|  \left( \B (w,u_m), u_m  \right) _{{L}^2_i(X)    } \right| 
					\leq 
					\frac{c}{\mu} 
					\| \B (w, u_m) \|^2_{{L}^2_i(X)  }
					+   \frac{\mu}{c} \| u_m \|^2_{{L}^2_i(X)  }
				\end{equation}
				with an arbitrary positive constant $c$.
				
				Next, using the Sobolev embedding theorem we have in each local card $ U_l $
				\begin{equation}\label{key233}
					\| w_{U_l} \|_{{L}^\infty (\mathbb{R}^n)} \leq c\|w_{U_l} \|_{{H}^{k+2} (\mathbb{R}^n)},
				\end{equation}
				where $ w_{U_l} $ is  representation of $ w $ in $ U_l $ and $ c $ is a positive constant. Then, under the Lemma \ref{equivalent} we get
				\begin{eqnarray}
					\label{eq.B.cont.1M}
					\int_0^t
					\| \B (w,u_m) \|^2_{{L}^2_i(X)   } ds \leq
					c\left( \int_0^t
					\!
					\| w \|^2_{{H}^2_i(X)   } \| \nabla_i u_m \|^2_{{L}^2 _i(X)  }
					ds
					+ \right. 
				\end{eqnarray}
				\[
				\left. \| w \|^2_{L^2 (I,{H}^3 _i(X)  )} \| u_m \|^2_{C (I,{L}^2_i(X)   )}\right).
				\]
				
				From Theorem \ref{t.exist.NS.lin.weak} and the Sobolev embedding theorem it
				follows that
				$$
				\| w \|^2_{L^2 (I,{H}^3_i(X)   )} \| u_m \|^2_{C (I,{L}^2_i(X)   )}
				\leq
				c (w) \left( \| u_0 \|^2_{L^2 (X)   }  + \| f \|^2_{L^2 (I,({V}^1_i)'   )}\right) 
				$$
				where $c (w)$ is a positive constant depending on
				$\| w \|^2_{L^2 (I,{H}^3_i(X)   )}$.
				
				Next, let  $u_{0,m}$ is an orthogonal 
				projection on the linear span  ${\mathcal L} (\{ b_j \}_{j \in {\mathbb N}} )$	in $V^{k'}_i$, then
				\[
				\lim_{m\to+\infty }\|u_0 - u_{0,m}\|_{{H}^{k'}_i (X)}= 0,\quad 
				\|u_{0,m}\|_{{H}^{k'}_i (X)} \leq \|u_{0}\|_{{H}^{k'}_i (X)}.
				\]
				for each $k' \in 
				{\mathbb Z}_+$.
				
				Then, after integration of \eqref{eq67} over the interval $[0,t]$ and taking into account \eqref{eq68},
				\eqref{eq69} and \eqref{eq.B.cont.1M} we have
				\begin{equation}
					\| u_m \|^2_{{L}^2_i (X)   }
					+ 
					\mu
					\int_0^t \| \nabla_i u_m \|^2_{{L}^2 _i (X)  } ds
					\leq 
					c\left( \| u_{0,m} \|^2_{{L}^2_i (X)   }\right. +
				\end{equation}
				\[
				\frac{2}{\mu} \| f \|^2_{L^2 (I,{L}^2_i (X)   )}
				\left. + c_{0,0}
				+ \frac{4}{\mu}
				\int_0^t
				\| w \|^2_{{ H}^2 _i (X)  } \| \nabla_i u_m \|^2_{{L}^2_i (X)   }
				ds\right) 
				\]
				for all $t \in [0,T]$.
				Here the constant $c_{0,0}$ depends on
				$\| w \|_{B^{0,2,1}_{i,\mathrm{vel}} (X_T)}$, 
				$\| u_0 \|^2_{L^2_i (X)   }$ and $ \| f \|^2_{L^2 (I,({V}^1_i)'   )} $
				only. It follows from the Gronwall's Lemma \ref{l.Groenwall}
				that
				\begin{equation}
					\|  u_m \|^2_{C (I,{L}^2_i (X)   )}
					+ \mu\, \| \nabla_i u_m \|^2_{L^2 (I,{L}^2_i (X)   )}
					\leq
					c_{1} (\mu,w,f,u_0)
				\end{equation}
				with a constant $c_{1} (\mu,w,f,u_0)$ depending on
				$\mu$
				and
				$\| w \|_{B^{0,2,1}_{i,\mathrm{vel}} (X_T)}$,
				$\| f \|_{B^{0,0,0}_\mathrm{i, for} (X_T)}$ and
				$\| u_0 \|_{V^1_i}$,
				only. 
				
				Now, substituting $ \nabla_i^{2r} u_m $ into \eqref{eq.NS.lin.weak} instead of $ v $ with $r \in {\mathbb Z}_+$ and $ u_m $ instead of $ u $ we have
				\begin{eqnarray}
					\label{eq.um.bound.OM.r}
					\lefteqn{
						\frac{d}{dt} \| \nabla_i^{r} u_m \|^2_{{L}_i^2(X)   }
						+ 2 \mu \| \nabla_i^{r+1} u_m \|^2_{{L}^2_i(X)    } = 
					}
					\nonumber
					\\
					& = &
					2 \left(  \nabla_i^{r-1} (\B (w,u_m) - f),
					\nabla_i^{r+1} u_m  \right) _{{L}^2_i(X)    }
				\end{eqnarray}
				for all $t \in [0,T]$.
				Furthermore, using H\"older inequality, we get
				\begin{equation}
					\label{eq.f.differ.est}
					2\left|  \left(  \nabla_i^{r-1} f,
					\nabla_i^{r+1} u_m  \right) _{{L}^2_i(X)    }\right| 
					\leq
					\frac{2}{\mu} \| \nabla_i^{r-1} f \|^2_{{L}^2_i(X)  }
					+ \frac{\mu}{2} \| \nabla_i^{r+1}   u_m \|^2 _{{L}^2_i(X)   } ,
				\end{equation}
				and similarly
				\begin{eqnarray}
					\label{eq.D.differ.est}
					2 \left|  \left(  \nabla_i^{r-1} \B (w,u_m),
					\nabla_i^{r+1} u_m  \right) _{{L}^2_i(X)    } \right| 
					\leq 
				\end{eqnarray}
				\[
				\frac{c}{\mu} 
				\| \nabla_i^{r-1} \B (w, u_m) \|^2_{{L}^2_i(X)  }
				+   \frac{\mu}{c} \| \nabla_i^{r+1} u_m \|^2_{{L}^2_i(X)  }
				\]
				with an arbitrary positive constant $c$.	
				
				Now, combining
				\eqref{eq.um.bound.OM.r} for $r=1$
				and
				\eqref{eq.f.differ.est},
				\eqref{eq.D.differ.est},
				\eqref{eq.B.cont.1M}
				with integration  over the interval $[0,t]$, we arrive at the estimate
				\begin{equation}
					\label{eq.um.bound.OM.1a}
					\| \nabla_i u_m (\cdot, t) \|^2_{{L}^2_i (X)   }
					+ 
					\mu
					\int_0^t \| \nabla^2_i u_m (\cdot, s) \|^2_{{L}^2 _i (X)  } ds
					\leq 
					c\left( \| \nabla_i u_{0,m} \|^2_{{L}^2_i (X)   }\right. +
				\end{equation}
				\[
				\frac{2}{\mu} \| f \|^2_{L^2 (I,{L}^2_i (X)   )}
				\left. + c_{0,0}
				+ \frac{4}{\mu}
				\int_0^t
				\| w \|^2_{{ H}^2 _i (X)  } \| \nabla_i u_m \|^2_{{L}^2_i (X)   }
				ds\right) 
				\]
				for all $t \in [0,T]$.
				Here the constant $c_{0,0}$ depends on
				$\| w \|_{B^{0,2,1}_{i,\mathrm{vel}} (X_T)}$, 
				$\| u_0 \|^2_{L^2_i (X)   }$ and $ \| f \|^2_{L^2 (I,({V}^1_i)'   )} $
				only.
				
				At this point inequality \eqref{eq.um.bound.OM.1a} and Gronwall's Lemma \ref{l.Groenwall}
				yield
				\begin{equation}
					\label{eq.um.bound.OM.strong.1}
					\| \nabla_i u_m \|^2_{C (I,{L}^2_i (X)   )}
					+ \mu\, \| \nabla^2_i u_m \|^2_{L^2 (I,{L}^2_i (X)   )}
					\leq
					c_{1} (\mu,w,f,u_0)
				\end{equation}
				with a constant $c_{1} (\mu,w,f,u_0)$ depending on
				$\mu$
				and
				$\| w \|_{B^{0,2,1}_{i,\mathrm{vel}} (X_T)}$,
				$\| f \|_{B^{0,0,0}_\mathrm{i, for} (X_T)}$ and
				$\| u_0 \|_{V^1_i}$,
				only. 
				
				Now, the Sobolev embedding theorem and H\"older inequality yield 
				\begin{eqnarray}
					\label{eq.B.cont.1M1}
					\int_0^t
					\| \nabla_i \B (w,u_m) \|^2_{{L}^2_i(X)   } ds \leq
					c\left(
					\int_0^t
					\| w \|^2_{{ H}^2 _i (X)  } \| \nabla^2_i u_m \|^2_{{L}^2 _i (X)  }
					ds
					+ \right. 
				\end{eqnarray}
				\[	+ 
				\left.  \| w \|^2_{L^2 (I,{H}^3 _i (X)  )} \| u_m \|^2_{C (I,{H}^1  _i (X) )}
				+ \| \nabla^2_i w \|^2_{C (I,{L}^2 _i (X)  )} \| u_m \|^2_{L^2 (I,{H}^2_i (X)   )}
				\right)
				\]
				for some positive constant $ c $.
				On combining
				\eqref{eq.um.bound.OM.r} for $r = 2$
				and
				\eqref{eq.B.cont.1M1},
				\eqref{eq.f.differ.est},
				\eqref{eq.D.differ.est}
				with integration  over the interval $[0,t]$ we obtain
				\begin{equation}
					\label{eq.um.bound.OM.1b}
					\| \nabla^2_i u_m (\cdot, t) \|^2_{{L}^2_i (X)   }
					+ \mu \int_0^t \| \nabla^3_i u_m (\cdot, s) \|^2_{{L}^2_i (X)   } ds
					\leq 
					\| \nabla^2_i u_{0,m} \|^2_{{L}^2_i (X)   }
					+ 
				\end{equation}
				\[ \frac{2}{\mu} \| \nabla_i f \|^2_{L^2 (I,{L}^2 _i (X)  )}
				+ c_{0,0}
				+ c_{1,0}
				\frac{2}{\mu}
				\int_0^t
				\| w \|^2_{{ H}^2 _i (X)  } \| \nabla^2_i u_m \|^2_{{L}^2 _i (X)  }
				ds
				\]
				\[
				+ 
				c_{1,0}
				\frac{2}{\mu}
				\Big( \| w \|^2_{L^2 (I,{H}^3 _i (X)  )} \| u_m \|^2_{C (I,{H}^1  _i (X) )}
				+ \| \nabla^2_i w \|^2_{C (I,{L}^2 _i (X)  )} \| u_m \|^2_{L^2 (I,{H}^2_i (X)   )}
				\Big).
				\]
				From inequalities
				\eqref{eq.um.bound.OM.strong.1},
				\eqref{eq.um.bound.OM.1b}
				and Gronwall's Lemma 
				\ref{l.Groenwall} 
				it follows readily that
				$$
				\| \nabla^2_i u_m \|^2_{C (I,{L}^2   )}
				+ \mu\, \| \nabla^3_i u_m \|^2_{L^2 (I,{L}^2 _i (X)  )}
				\leq
				c_{2} (\mu,w,f,u_0),
				$$
				where $c_{2} (\mu,w,f,u_0)$ is a constant depending on
				$\mu$
				and
				$\| w \|_{B^{0,2,1}_{i,\mathrm{vel}} (X_T)}$,
				$\| f \|_{B^{0,0,0}_{i,\mathrm{vel}} (X_T)}$ and
				$\| u_0 \|_{V^2_i}$,
				only.
				
				Assume that the sequence $\{ u_m \}$ is bounded in the space
				$B^{k,2,1}_{i,\mathrm{vel}} (X_T)$,
				for given data $(f,u_0) \in B^{k,0,0}_{i,\mathrm{for}} (X_T) \times V_i^{k+2}$,
				with $k = k' \in \mathbb N$,
				i.e.,
				\begin{equation}
					\label{eq.um.bound.OM.strong.k.prime}
					\| \nabla^{k''}_i u_m \|^2_{C (I,{L}^2 _i (X)   )}
					+ \mu\, \| \nabla^{k''+1}_i u_m \|^2_{L^2 (I,{L}^2 _i (X)   )}
					\leq
					c_{k''} (\mu,w,f,u_0),
				\end{equation}
				if $0 \leq k'' \leq k'+2$, where the constants $c_{k''} (\mu,w,f,u_0)$ depend on
				$\mu$
				and the norms
				$\| w \|_{B^{k,2,1}_{i,\mathrm{vel}} (X_T)}$,
				$\| f \|_{B^{k,0,0}_{i,\mathrm{for}} (X_T)}$,
				$\| u_0 \|_{V^{k+2}_i}$
				but not on $m$.
				Then, combining
				\eqref{eq.f.differ.est},
				\eqref{eq.D.differ.est}
				with integration over the time interval $[0,t]$, we get
				\begin{eqnarray}
					\label{eq.um.bound.OM.1k}
					\lefteqn{
						\| \nabla^{k'+3}_i u_m  (\cdot,t) \|^2_{{L}^2 _i (X)   }
						+ \mu \int_0^t \| \nabla^{k'+4}_i u_m  (\cdot,s) \|^2_{{L}^2 _i (X)   } ds
					}
					\nonumber
					\\
					& \leq 
					\| \nabla^{k'+3}_i u_{0,m} \|^2_{{L}^2 _i (X)   }
					+ \| \nabla^{k'+2}_i f \|^2_{L^2 (I,{L}^2 _i (X) )}
					+ \frac{2}{\mu} \| \nabla^{k'+2}_i \B (w,u_m) \|^2_{L^2 (I,{L}^2_i (X)  )}.
					\nonumber
					\\
				\end{eqnarray}
				In this way we need to evaluate the last summand on the right-hand side of 
				\eqref{eq.um.bound.OM.1k}.
				For all $ 0\leq k'\leq k $, similarly to \eqref{eq25}, we have
				\begin{equation}\label{ocenB}
					\| \nabla^{k'+2}_i \B (w,u_m) \|^2_{L^2 (I,{L}^2 _i(X)   )} \leq 
				\end{equation}
				\[
				c_{0}\,
				\frac{2}{\mu}
				\int_0^t
				\| \nabla^{k'+3}_i u_m (\cdot,s) \|^2_{{L}^2 _i(X)  }
				\| w (\cdot,s) \|^2_{{H}^2 _i(X)  }
				ds + 
				\]
				\[
				\frac{2}{\mu}
				\sum_{j=1}^{k'+1}
				c_{j}\,
				\| \nabla^{k'+3-j}_i u_m \|^2_{C (I,{L}^2_i(X)   )}
				\| w \|^2_{L^2 (I,{H}^{k'+ 3} _i(X)  )} + 
				\]
				\[
				\frac{2}{\mu}
				\sum_{j=1}^{k'+1}
				c_{j}\,
				\| \nabla^{k'+3-j} _iw \|^2_{C (I,{L}^2 _i(X)  )}
				\| u_m \|^2_{L^2 (I,{H}^{k'+3} _i(X)  )} + 
				\]
				\[
				c_{0}\,
				\frac{2}{\mu}\,
				\| w \|^2_{C (I,{H}^{k'+3} _i(X)  )}
				\| \nabla^{k'+2}_i u_m \|^2_{L^2 (I,{L}^{2}_i(X)   )}+
				\]
				with positive constants $ c_{j} $.
				All terms on the right-hand side of this inequality can be estimated due to the inductive
				assumption of \eqref{eq.um.bound.OM.strong.k.prime} and the Sobolev embedding theorem.
				From
				\eqref{ocenB},
				\eqref{eq.um.bound.OM.strong.k.prime} and
				\eqref{eq.um.bound.OM.1k},
				it follows that
				\begin{eqnarray}
					\label{eq.um.bound.OM.1k.A}
					\| \nabla^{k'+3}_i u_m  (\cdot, t) \|^2_{{L}^2 _i(X)  }
					+ \mu \int_0^t \| \nabla^{k'+4}_i u_m  (\cdot,s) \|^2_{{L}^2 _i(X)  } ds
					\, \leq \,
					\| \nabla^{k'+3}_i u_{0} \|^2_{{L}^2_i(X)   }+ 
				\end{eqnarray}
				\[
				\| \nabla^{k'+2}_i f \|^2_{L^2 (I,{L}^2 _i(X)  )}
				+ c_{k'+2,0}\,
				\frac{2}{\mu}
				\int_0^t
				\| \nabla^{k'+3}_i u_m (\cdot,s) \|^2_{{L}^2 _i(X)  }
				\| w (\cdot,s) \|^2_{{H}^2 _i(X)  }
				ds + 
				\]
				\[
				R_{k'+3} (\mu,w,f,u_0),
				\]
				for all $t \in [0,T]$, the remainder $R_{k'+3} (w,f,u_0)$ depends on
				$\mu$
				and
				$\| w \|_{B^{k'+1,2,1}_{i,\mathrm{vel}} (X_T)}$,
				$\| f \|_{B^{k'+1,0,0}_{i,\mathrm{for}} (X_T)}$,
				$\| u_0 \|_{V^{k'+2}_i}$,
				only.
				
				As before,
				\eqref{eq.um.bound.OM.strong.k.prime},
				\eqref{eq.um.bound.OM.1k.A}
				and Gronwall's Lemma 
				\ref{l.Groenwall} 
				yield
				$$
				\| \nabla^{k'+3}_i u_m \|^2_{C (I,{L}^2 _i(X)   )}
				+ \mu\, \| \nabla^{k'+4}_i u_m \|^2_{L^2 (I,{L}^2 _i(X)   )}
				\leq
				c_{k'+3} (\mu,w,f,u_0),
				$$
				the constant $c_{k'+3} (\mu,w,f,u_0)$ depends on
				$\mu$
				and
				$\| w \|_{B^{k'+1,2,1}_{i,\mathrm{vel}} (X_T)}$,
				$\| f \|_{B^{k'+1,0,0}_{i,\mathrm{for}} (X_T)}$ and
				$\| u_0 \|_{V_i^{k'+3}}$
				but not on the index $m$.
				When combined with the induction hypothesis of \eqref{eq.um.bound.OM.strong.k.prime}, the
				latter estimate implies that the assertion of the lemma is true for all $k \in {\mathbb Z}_+$.
			\end{proof}

			It follows from the Lemma \eqref{l.OM.bound} that the sequence $\{ u_m\}$
			is bounded
			in the space
			$C (I,{H}^{k+2}_i (x)   )\cap L^2 (I,{H}^{k+3} _i (x)  )$
			if
			the data $(f,u_0)$ belong to $B^{k,0,0}_\mathrm{i, for} (X_T) \times V_i^{k+2}$.
			Hence it follows that we may extract a subsequence $\{ u_{m'} \}$, such that
			
			1)
			for any $j$ satisfying $j \leq k+3$, the sequence
			$\{ \nabla_i^j  u_{m'} \}$
			converges weakly in $L^2 (I,{L}^{2}  _i (x) )$.
			
			2)
			the sequence 
			$\{   u_{m'} \} $
			converges $^{\ast}$-weakly in $L^\infty (I,{H}^{k+2} _i (x)  ) \cap L^2 (I,{H}^{k+3} _i (x)  )$ to an element
			$u$.
			
			On the other hand, applying Lemma \ref{l.Groenwall}  to the
			inequality \eqref{eq.um.bound.1} we see that 
			the sequence  $\{\partial_t u_m\}$ is bounded in the space 
			$L^2 (I,{L}^2_i(X))$. In particular, 
			we may extract a subsequence $\{\partial_t u_{m'}\}$ such that 
			$\{ \partial_t  u_{m'}\} $  
			converges weakly in $L^2 (I,V^0_i)$ to an element $\tilde 
			u \in L^2 (I,V^0_i)$. 
			
			By the very construction and Theorem \ref{t.exist.NS.lin.weak}, 
			the section $u$ is the unique solution to \eqref{eq.NS.lin.weak}  
			from the space 	$L^\infty (I,V_i^{k+2}) \cap L^2 (I,V_i^{k+3}) \cap C (I,V_i^0)$. Hence,  Lemma \ref{l.Lions} yields
			$\nabla_i^j u \in C^2 (I,{L}^{2} _i(X)  )$
			if $j \leq k+2$.
			Moreover, by \eqref{eq.B.cont.6}, the section $\B (w, u)$ belongs to
			$C (I,{H}^k_i(X)   ) \cap L^2 (I,{H}^{k+1}_i(X)   )$.
			
			Actually,   \eqref{eq.NS.lin.weak} imply that 
			\begin{equation} \label{eq.um.bound.OM.1t}
				\partial_t u  = -\mu \Delta^i u +\mathrm {P}^i (\B(w, u)-f) \mbox{ in } 
				X_T
			\end{equation}
			in the sense of distributions. According to Lemma \ref{p.nabla.Bochner}, the projection 
			$\mathrm {P}^i$ maps
			$C (I,{H}^k_i (X)   ) \cap L^2 (I,{H}^{k+1}  _i (X) )$
			continuously into
			$C (I,V_i^k) \cap L^2 (I,V_i^{k+1})$. Then the section $\partial_t u$ belongs to
			$C (I,V_i^k) \cap L^2 (I,V_i^{k+1})$.
			
			We have thus proved that \eqref{eq.NS.lin.weak} admits a unique solution
			$u \in  B^{k,2,1}_{i,\mathrm{vel}} (X_T)$
			for any data $(f,u_0) \in B^{k,0,0}_{i,\mathrm{for}} (X_T) \times V_i^{k+2}$.
			
			Now, it follows from Lemma \ref{proector} that 
			\[
			(I - \mathrm {P}^i) (f - \B (w, u)) = A^{i-1}(A^{i-1})^*(f - \B (w, u))
			\]
			and then the Corollary \ref{p.nabla.Bochner} implies
			that there is a unique function $p\in B
			^{k+1,0,0}_{i-1,\mathrm{pre}} (X_T)$ 
			such that 
			\begin{equation} \label{eq.nabla.p}
				A^{i-1} p= (I- \mathrm {P}^i) (f- \B(w, u)) \mbox{ in } X_T.
			\end{equation}
			Adding \eqref{eq.um.bound.OM.1t} to \eqref{eq.nabla.p} we conclude that 
			the pair 
			\[
			(u,p) \in B^{k,2,1}_{i,\mathrm{vel}} (X_T)\times 
			B^{k+1,0,0}_{i-1,\mathrm{pre}} (X_T)
			\]
			is the unique solution to 
			\eqref{eq.NS.lin}  related to the datum $(f,u_0) \in B^{k,0,0}
			_{i,\mathrm{for}} (X_T) \times V_i^{2+k}$. This implies the surjectivity of the 
			mapping ${\mathcal A}_{w,i}$ for $ s=1 $ and for any $k \in {\mathbb Z}_+$.
			
			We finish the proof of the theorem with induction in $s \in \mathbb N$.
			More precisely, assume that the assertion of the theorem concerning the surjectivity of the mapping $\mathcal{A}_{w,i}$ holds for some $s = s' \in \mathbb N$ and any $k \in {\mathbb Z}_+$. Let
			$(f,u_0) \in B^{k,2s',s'}_{i,\mathrm{for}} (X_T) \times V_i^{2 (s'+1)+k}$.
			It is clear that
			\[
			B^{k,2s',s'}_{i,\mathrm{for}} (I) \times V_i^{2 (s'+1)+k}
			\hookrightarrow
			B^{k+2,2(s'-1),s'-1}_{i,\mathrm{for}} (X_T) \times V_i^{2s'+k+2},
			\]
			and we see that according to the induction assumption there is a unique solution $(u,p)$ to \eqref{eq.NS.lin} which belongs to
			$ B^{k+2,2s',s'}_{i,\mathrm{vel}} (X_T) \times B^{k+3,2(s'-1),s'-1}_{i-1,\mathrm{pre}} (X_T)$.
			
			By Theorem \ref{l.NS.cont.0} and Lemma \ref{proector}, the sections
			$\varDelta^i u$,
			$\B(w,u)$ and
			$\mathrm {P}^i (f - \B(w,u))$
			belong to
			$B^{k,2s',s'}_{i,\mathrm{for}} (X_T)$,
			and so the derivative $\partial_t u$ is in this space, too, because of \eqref{eq.um.bound.OM.1t}.
			As a consequence, \eqref{eq.nabla.p} implies
			$A^{i-1} p \in B^{k,2s',s'}_{i,\mathrm{for}} (X_T)$,
			and so
			$p \in B^{k+1,2s',s'}_{i-1,\mathrm{pre}} (X_T)$.
			
			Thus, the pair $(u,p)$ actually belongs to
			$B^{k,2(s'+1),s'+1}_{i,\mathrm{vel}} (X_T) \times B^{k+1,2s',s'}_\mathrm{i-1, pre} (X_T)$,
			i.e., the mapping $\mathcal{A}_{w,i}$ of \eqref{eq.map.Aw} is surjective for all $k \in {\mathbb Z}_+$
			and $s \in {\mathbb N}$.
		\end{proof}
		
		Finally, as the mapping $\mathcal{A}_{w,i}$ is bijective and continuous, the continuity of the 
		inverse $\mathcal{A}^{-1}_{w,i}$ follows from the inverse mapping theorem for Banach spaces.
	\end{proof}
	
	Now we may formulate the main results of this paper.
	
	\begin{theorem}
		\label{t.open.NS.short}
		Let  $n\geq 2$,  $s \in \mathbb N$ and $k \in {\mathbb Z}_+$, $2s+k >n/2$. 
		Then \eqref{eq3} induces an injective continuous nonlinear mapping
		\begin{equation}
			\label{eq.map.A}
			\mathcal{A}_i :
			B^{k,2s,s}_{i,\mathrm{vel}} (X_T) \times B^{k+1,2(s-1),s-1}_{i-1,\mathrm{pre}} 
			(X_T) \to
			B^{k,2(s-1),s-1}_{i,\mathrm{for}} (X_T) \times H_i^{2s+k}
		\end{equation}
		which is moreover open.
	\end{theorem}
	
	\begin{proof}
		Indeed, the continuity of the mapping $\mathcal{A}_i$ is clear from 
		Theorem \ref{l.NS.cont.0}.
		
		Moreover, suppose that
		$$
		\begin{array}{rcl}
			(u,p)
			& \in
			&   B^{k,2s,s}_{i,\mathrm{vel}} (X_T) \times 
			B^{k+1,2(s-1),s-1}_{i-1,\mathrm{pre}} (X_T),
			\\
			\mathcal{A}_i (u,p)
			\, = \, (f,u_0)&\in
			&   B^{k,2(s-1),s-1}_{i,\mathrm{for}} (X_T) \times H_i^{k+2s}.
		\end{array}
		$$
		
		Let us show that Problem \eqref{eq3} has at most one solution $ (u,p) $ in the space $ B^{k,2s,s}_{i,\mathrm{vel}} (X_T) \times 
		B^{k+1,2(s-1),s-1}_{i-1,\mathrm{pre}} (X_T) $. Indeed, let $ (u', p') $ and $(u'',p'')$ be any two solutions to \eqref{eq3} from the declared
		function space, i.e. $\mathcal{A}_i (u',p') = \mathcal{A}_i (u'',p'')$
		and sections $ u = u'-u'' $ and $ p = p'-p'' $ satisfies \eqref{eq3} with zero data $ (f,u_0) = (0,0) $. Moreover, as the left side of \eqref{eq3} is integrable with square we have 
		\begin{equation*}
			\frac{d}{dt} \|u\|^2_{{L}_{i}^2(X)}
			+ 2\mu \|A^i u\|^2_{{L}_{i}^2(X)} = \Big(\left( \B (u'',u'') -  \B (u',u')\right) , u\Big)_{{L}_{i}^2(X)} ,
		\end{equation*}
		with the condition $ u (\cdot,0) = 0 $. Hence it follows from the \eqref{eq9}, Lemma \ref{l.Groenwall} and H\"older's inequality that $ u\equiv 0 $. On the other hand, Corollary 
		\ref{p.nabla.Bochner} implies $ p'=p'' $.	So, the operator $\mathcal{A}_i$ of \eqref{eq.map.A} is injective.
		
		Finally, it easy to see that the Frech\'et derivative
		$\mathcal{A}'_i {(w,p_0)}$
		of the nonlinear mapping $\mathcal{A}$ at an arbitrary point
		$$
		(w,p_0)
		\in   B^{k,2s,s}_{i,\mathrm{vel}} (X_T) \times 
		B^{k+1,2(s-1),s-1}_{i-1,\mathrm{pre}} (X_T)
		$$
		coincides with the continuous linear mapping $\mathcal{A}_{w,i}$ of \eqref{eq.map.Aw}.
		By \eqref{embb} and Theorem \ref{t.exist.NS.lin.strong}, $\mathcal{A}_{w,i}$ is an invertible continuous linear
		mapping from
		$B^{k,2s,s}_{i,\mathrm{vel}} (X_T) \times   
		B^{k+1,2(s-1),s-1}
		_{i-1,\mathrm{pre}} (X_T)$ to
		$  B^{k,2(s-1),s-1}_{i,\mathrm{for}} (X_T) \times H_i^{k+2s}$.
		Both the openness of the mapping $\mathcal{A}_i$ and the continuity of its local inverse mapping
		now follow from the implicit function theorem for Banach spaces, 
		see for instance  \cite[Theorem 5.2.3, p.~101]{Ham82}.
	\end{proof}

	For the de Rham complex over the torus ${\mathbb T}^3$  at the degree $i=1$ 
	Theorem \ref{t.open.NS.short} was proved in \cite{ShT20}; actually 
	this situation corresponds to the Navier-Stokes equations for incompressible 
	fluid in the periodic setting, see \cite{Temam95}.

	It is worth to note that for an  Existence Theorem related to even  weak (distributional) 
	solutions to \eqref{eq3} one should necessarily assume that the bilinear forms $M_{i,1}$ 
	have additional properties. For example, in the above case for the de Rham complex this is 
	the vanishing property of the so-called trilinear form, see  \cite{Ladyz}, \cite{Lion69}, 
	\cite{Temam79}. This means that the open mapping theorem is only a first step toward an 
	Existence Theorem for regular solutions to \eqref{eq3}. 
	
	\bigskip
	
	\textit{The work was supported by the Foundation for the Advancement of Theoretical Physics and Mathematics "BASIS".}

\end{document}